\documentclass[11pt]{article}
\usepackage{amsthm, amsmath, amssymb, amsfonts, url, booktabs, tikz, setspace, fancyhdr, amsbsy}
\usepackage[margin = 1in]{geometry}
\usepackage{esint}
\usepackage{verbatim}

\newtheorem{theorem}{Theorem}[section]
\newtheorem{proposition}[theorem]{Proposition}
\newtheorem{lemma}[theorem]{Lemma}
\newtheorem{corollary}[theorem]{Corollary}

\theoremstyle{definition}
\newtheorem{definition}[theorem]{Definition}

\theoremstyle{remark}
\newtheorem*{remark}{Remark}

\newcommand{\norm}[1]{\left\lVert#1\right\rVert}

\newcommand{\R}{\mathbb{R}}
\newcommand{\V}{\mathbb{V}}
\newcommand{\Q}{\mathbb{Q}}
\newcommand{\Z}{\mathbb{Z}}
\newcommand{\defeq}{\mathrel{\mathop:}=}
\newcommand{\eqdef}{\mathrel{\mathop=}:}

\newcommand{\U}{\mathcal{U}}
\newcommand{\G}{\mathcal{G}}
\newcommand{\uu}{\boldsymbol{u}}
\newcommand{\DD}{\boldsymbol{D}}
\newcommand{\UU}{\boldsymbol{U}}
\newcommand{\VV}{\boldsymbol{V}}
\newcommand{\vv}{\boldsymbol{v}}
\newcommand{\SSS}{\boldsymbol{S}}
\newcommand{\q}{\boldsymbol{q}}
\newcommand{\ww}{\boldsymbol{w}}
\newcommand{\hh}{\boldsymbol{h}}
\newcommand{\dx}{\,\mathrm{d}x}
\newcommand{\dy}{\,\mathrm{d}y}
\numberwithin{equation}{section}

\def\ocirc#1{\ifmmode\setbox0=\hbox{$#1$}\dimen0=\ht0
    \advance\dimen0 by1pt\rlap{\hbox to\wd0{\hss\raise\dimen0
    \hbox{\hskip.2em$\scriptscriptstyle\circ$}\hss}}#1\else
    {\accent"17 #1}\fi}

\begin{document}

\title{Finite element approximation of an incompressible chemically reacting non-Newtonian fluid}

\author{Seungchan Ko\thanks{Mathematical Institute, University of Oxford, Andrew Wiles Building, Woodstock Road, Oxford OX2 6GG, UK. Email: \tt{seungchan.ko@maths.ox.ac.uk}} , ~Petra Pust{\v{e}}jovsk{\'a}\thanks{TU Munich, Chair of Numerical Mathematics, Boltzmannstrasse 3, D-85748 Garching b. M\"unchen, Germany. Email: \tt{petra.pustejovska@gmail.com}}
,~and ~Endre S\"uli\thanks{Mathematical Institute, University of Oxford, Andrew Wiles Building, Woodstock Road, Oxford OX2 6GG, UK. Email: \tt{endre.suli@maths.ox.ac.uk}}}

\date{~}

\maketitle

~\vspace{-1.5cm}

\begin{abstract}
We consider a system of nonlinear partial differential equations modelling the steady motion of an incompressible non-Newtonian fluid,
which is chemically reacting. The governing system consists of a steady convection-diffusion equation for the concentration and the
generalized steady Navier--Stokes equations, {\color{black}{where the viscosity coefficient is a power-law type function of the shear-rate, and the coupling between the equations results from the concentration-dependence of the power-law index.}}
This system of nonlinear partial differential equations arises in mathematical models of the synovial fluid found in the cavities of moving joints. We construct a finite element approximation of the model and perform the mathematical analysis of the numerical method in the case of two space dimensions. Key technical tools include discrete counterparts of the Bogovski\u{\i} operator, De Giorgi's regularity theorem in two dimensions, and the Acerbi--Fusco Lipschitz truncation of Sobolev functions, in function spaces with variable integrability exponents.
\end{abstract}

\smallskip


\smallskip


\begin{section}{Introduction}
During the past decade the mathematical study of non-Newtonian fluids has become an active field of research, stimulated by
the wide range of scientific and industrial problems in which they arise. Examples of non-Newtonian fluids include biological
fluids (such as mucus, blood, and various polymeric solutions), as well as numerous fluids of significance in engineering,
food industry, cosmetics, and agriculture.
{\color{black}{In this paper, we
shall investigate a system of nonlinear partial differential
equations (PDEs) modelling the motion of the synovial fluid (a biological fluid found in the cavities of moving joints)
in a steady shear experiment.}} From the rheological viewpoint, the synovial fluid consists of ultrafiltrated blood plasma diluting
a particular polysaccharide,  called \textit{hyaluronan}.
Though one could model the solution using mixture theory, we shall restrict ourselves to the situation where the solution
can be described as a single-constituent fluid. This perspective is fairly reasonable because the mass concentration of
hyaluronan is negligible, and even if molecules of hyaluronan are accumulated locally, the mass concentration does not exceed 2\%.
Nevertheless, {\color{black}{we still need to consider the experimentally observed
chemical properties of the fluid. In fact, it was already observed in viscosimetric experiments performed in the early 1950s that the synovial
fluid has a strong shear-thinning property, depending on the concentration of hyaluronan in the solution.
Explicitly, the viscosity of the fluid is a function of the concentration as well as of the shear rate.
Therefore, from the viewpoint of mathematical modelling a power-law-like model, where the power-law index is concentration-dependent,
seems reasonable.}}

Denoting by $c$ the concentration of hyaluronan in the solution and by $\boldsymbol{D}\uu:=\frac{1}{2}(\nabla \uu
+ (\nabla\uu)^{\rm T})$, the symmetric gradient of the velocity field $\uu$, it was observed in laboratory experiments
(see \cite{exp}) that {\color{black}{the effect of concentration and the shear rate on the viscosity are not separated
(as, for instance, $\nu(c,|\boldsymbol{D}\uu|^2)\sim f(c)\tilde{\nu}(|\boldsymbol{D}\uu|^2)$), but that the  concentration of hyaluronan affects the level of
shear thinning. For zero concentration, the viscosity becomes constant,
corresponding to the fact that the fluid is composed only of ultrafiltrated blood plasma,
exhibiting properties of a Newtonian fluid. If the concentration of hyaluronan increases,
the fluid displays higher apparent viscosity and, in fact, it thins the shear more markedly. Therefore a new power-law-like model
of the synovial fluid was proposed in \cite{exp2}, where the power-law index was considered to be a function of the concentration.
This new model describes the viscous properties of the synovial fluid more accurately,}} and it naturally reflects the fact that non-Newtonian effects diminish as the concentration of hyaluronan decreases.

{\color{black}{Based on the discussion above, we shall investigate a system of equations describing the motion of a
shear-thinning fluid with a}} nonstandard growth condition on the viscosity. More precisely, we shall consider
the incompressible generalized Navier--Stokes equations with a power-law-like viscosity where the power-law index
is not fixed, but depends on the concentration. To close the system, {\color{black}{we shall assume that the concentration satisfies a convection-diffusion equation. The resulting system of partial differential
equations is therefore fully coupled.}}

{\color{black}{In other words, we consider the following system of PDEs:}}
\begin{alignat}{2}
\text{div}\,{\uu}&=0\qquad &&\mbox{in $\Omega$},\label{eq1}\\
\text{div}\,(\uu\otimes \uu)-\text{div}\,\SSS(c,\DD\uu)&=-\nabla p+\boldsymbol{f}\qquad &&\mbox{in $\Omega$},\label{eq2}\\
\text{div}\,(c\uu)-\text{div}\,\q_c(c,\nabla c,\DD\uu)&=0\qquad &&\mbox{in $\Omega$},\label{eq3}
\end{alignat}
in a bounded open Lipschitz domain $\Omega\subset\R^d$, $d \in \{2,3\}$, where $\uu:\overline{\Omega}\rightarrow\R^d$, $p:\Omega\rightarrow\R$, $c:\overline{\Omega}\rightarrow\R_{\geq0}$ are the velocity, pressure and concentration fields, respectively. In the present context, $\boldsymbol{f}:\Omega\rightarrow\R^d$ denotes a given external force, $\DD\uu$ denotes the symmetric velocity gradient, i.e., $\DD\uu=\frac{1}{2}(\nabla \uu+(\nabla \uu)^{\rm T})$, and $\SSS(c,\DD\uu)$ and $\q_c(c,\nabla c,\DD\uu)$ are the extra stress tensor and the diffusive flux respectively. To complete the problem, we prescribe the following Dirichlet boundary conditions
\begin{equation}\label{bc}
\uu=\boldsymbol{0},\qquad c=c_d\qquad\text{on}\,\,\partial\Omega,
\end{equation}
where $c_d\in W^{1,q}(\Omega)$ for some $q>d$ and $c_d\geq0$ a.e. on $\Omega$. Thanks to the Sobolev embedding $W^{1,q}(\Omega)\hookrightarrow C(\overline{\Omega})$, we can therefore define
\[c^-\defeq\min_{x\in\overline{\Omega}}c_d\,\,\,\,\,\text{and}\,\,\,\,\,c^+\defeq\max_{x\in\overline{\Omega}}c_d.\]
We shall assume that the stress tensor $\SSS:\R_{\geq 0}\times\R^{d\times
d}_{\rm sym}\rightarrow\R^{d\times d}_{\rm sym}$ is a continuous function
satisfying the following growth, strict monotonicity and coercivity
conditions, respectively: there exist positive constants $C_1$, $C_2$ and $C_3$ such that
\begin{equation}\label{S1}
|\SSS(\xi,\boldsymbol{B})|\leq C_1(|\boldsymbol{B}|^{r(\xi)-1}+1),
\end{equation}
\begin{equation}\label{S2}
(\SSS(\xi,\boldsymbol{B_1})-\SSS(\xi,\boldsymbol{B_2}))\cdot(\boldsymbol{B_1}-\boldsymbol{B_2})>0\,\,\,\text{for}\,\,\boldsymbol{B_1}\neq
\boldsymbol{B_2},
\end{equation}
\begin{equation}\label{S3}
\SSS(\xi,\boldsymbol{B})\cdot \boldsymbol{B}\geq C_2(|\boldsymbol{B}|^{r(\xi)}+|\SSS|^{r'(\xi)})-C_3,
\end{equation}
where $r:\R_{\geq0}\rightarrow\R_{\geq0}$ is a H\"older-continuous function with $1<r^-\leq r(\xi)\leq r^+<\infty$ and $r'(\xi)$ is defined as its H\"older conjugate, $\frac{r(\xi)}{r(\xi)-1}$.
We further assume that the concentration flux vector
$\q_c(\xi,\boldsymbol{g},\boldsymbol{B}):\R_{\geq 0}\times\R^d\times\R^{d\times d}_{\rm{sym}}\rightarrow\R^d$ is a
continuous function, which is linear with respect to $\boldsymbol{g}$, and additionally satisfies the following inequalities: there exist positive constants $C_4$ and $C_5$ such that
\begin{align}
|\q_c(\xi,\boldsymbol{g},\boldsymbol{B})|&\leq C_4|\boldsymbol{g}|,\label{q1}\\
\q_c(\xi,\boldsymbol{g},\boldsymbol{B})\cdot \boldsymbol{g}&\geq C_5|\boldsymbol{g}|^2.\label{q2}
\end{align}

The prototypical examples we have in mind are of the following form:
\[\SSS(c,\DD\uu)=\nu(c,|\DD\uu|)\DD\uu,\qquad \q_c(c,\nabla c,\DD\uu)=\boldsymbol{K}(c,|\DD\uu|)\nabla c,\]
{\color{black}{where the viscosity $\nu(c,|\DD\uu|)$, depending on the concentration and on the shear-rate, is of the form:}}
\[\nu(c,|\DD\uu|)\sim\nu_0(\kappa_1+\kappa_2|\DD\uu|^2)^{\frac{r(c)-2}{2}},\]
where $\nu_0,\kappa_1,\kappa_2$ are positive constants.

{\color{black}{The coupled system of generalized Navier--Stokes equations and a convection-diffusion-reaction
equation with diffusion coefficient depending on both the shear rate and the
concentration was first studied in \cite{BMR2008},
{\color{black}{where, however, the shear-thinning index was a fixed constant and the influences
of the concentration and the shear
rate were separated. There, the authors considered the unsteady model and established the long-time existence
of weak solutions subject to large initial data with a}}
constant $r>\frac{9}{5}$ exploiting an $L^{\infty}$ truncation method.

Here we are faced with a model where the shear-thinning index is not a fixed constant or a fixed function, but is
concentration-dependent. The mathematical analysis of the model \eqref{eq1}--\eqref{q2}, where the power law-index
depends on the concentration, was initiated in \cite{BP2013}
by using generalized monotone operator theory for $r^->\frac{3d}{d+2}$.
Recently, in \cite{BP2014}, the authors succeeded in lowering the bound on
$r^-$ to $\frac{d}{2}$ and proved the H\"older-continuity of the concentration. It was emphasized in
\cite{BP2014} that the bound $r^->\frac{d}{2}\geq\frac{2d}{d+2}$
{\color{black}{ensures that one can guarantee H\"older-continuity of $c$ by using De Giorgi's method.
In fact, according to the results in \cite{DMS2008}, in the framework of variable-exponent spaces,
at least some regularity of the power-law exponent is required, not only for the Lipschitz truncation
method, which strongly relies on the continuity of the exponent, but also for the purpose of extending
classical Sobolev embedding theorems, various functional inequalities, and the boundedness of the maximal
operator, to variable-order counterparts of classical function spaces; see the next section for more details.}}}}

As for the finite element approximation of the model \eqref{eq1}--\eqref{q2}, no results have been established so far. We mention, however, some related developments: recently, in \cite{DKS2013}, using various weak compactness techniques, such as Chacon's biting lemma, Young measures,  and a new finite element counterpart of the Lipschitz truncation method, the Diening
et al. proved the convergence of the finite element approximation of a general class of steady incompressible non-Newtonian fluid flow models (not coupled to a convection-diffusion equation, though,) where the viscous stress tensor and the rate-of-strain tensor {\color{black}{were related through a, possibly discontinuous, maximal monotone graph.}}
{\color{black}{In \cite{R2004}, R{\ocirc{u}}{\v{z}}i{\v{c}}ka considered electro-rheological models with a fixed power-law exponent; a fully-implicit time discretization was developed and an error estimate was obtained. Concerning PDEs
with nonlinearities involving a variable exponent, in \cite{DAA}, Duque et al.
focused on a porous medium equation with a variable exponent, which was a given function, and they established the convergence of a sequence of finite element approximations to the problem. Furthermore, in \cite{BBD}, electro-rheological fluids were studied, where
the stress tensor was of power-law type with a variable power-law exponent; a discretization of
the problem was constructed and the convergence of the sequence of discrete solutions to a weak solution was shown.}}

In this paper we consider the construction of a finite element approximation of the system of nonlinear partial differential  equations \eqref{eq1}--\eqref{q2} and, motivated by the ideas in \cite{BP2014},  we develop the convergence analysis of this numerical method in the case of variable-exponent spaces in a two-dimensional domain. We note that the extension of the results of this paper to the case of three space dimensions is beyond the reach of the analysis developed here, because there is currently no finite element counterpart of De Giorgi's estimate for the three-dimensional nonlinear convection-diffusion equation satisfied by the concentration $c$. Nevertheless, at least initially, we shall admit $d \in \{2,3\}$. Subsequently we shall restrict ourselves to the case of $d=2$. Also, as no uniqueness result is currently available for weak solutions of the problem under consideration, we can only show that a subsequence of the sequence of numerical approximations converges to \textit{a} weak solution of the problem.

\end{section}

\begin{section}{Notation and auxiliary results}
In this section, we introduce some function spaces and preliminaries, which will be used throughout. Let $\mathcal{P}$ be the set of all measurable functions $r:\Omega\rightarrow[1,\infty]$; we shall call the function $r\in\mathcal{P}(\Omega)$ a variable exponent. We also define $r^-\defeq$ ess $\inf_{x\in\Omega}r(x)$, $r^+\defeq$ ess $\sup_{x\in\Omega}r(x)$ and for simplicity, we only consider the case
\begin{equation}\label{pcondition}
1<r^-\leq r^+<\infty,
\end{equation}
as $r^-=1$ and $r^+=\infty$ are of no physical relevance in the PDE model under consideration here.

Since we are considering the case of a power-law index depending on concentration, we need to work in Lebesgue and Sobolev spaces with variable exponents. To be more precise, we introduce the following variable-exponent Lebesgue spaces, equipped with the corresponding Luxembourg norms:
\begin{align*}
L^{r(\cdot)}(\Omega)&\defeq \left\{u\in L^1_{\rm{loc}}(\Omega):\int_{\Omega}|u(x)|^{r(x)}\dx<\infty\right\},\\
\norm{u}_{L^{r(\cdot)}(\Omega)}=\norm{u}_{r(\cdot)}&\defeq \inf\left\{\lambda>0:\int_{\Omega}\bigg|\frac{u(x)}{\lambda}\bigg|^{r(x)}\dx\leq1\right\}.
\end{align*}
In the same way, we introduce the following variable-exponent Sobolev spaces
\begin{align*}
W^{1,r(\cdot)}(\Omega)&\defeq \left\{u\in W^{1,1}(\Omega)\cap L^{r(\cdot)}(\Omega):|\nabla u|\in L^{r(\cdot)}\right\},\\
\norm{u}_{W^{1,r(\cdot)}(\Omega)}=\norm{u}_{1,r(\cdot)}&\defeq \inf\left\{\lambda>0:\int_{\Omega}\left[\bigg|\frac{u(x)}{\lambda}\bigg|^{r(x)}+\bigg|\frac{\nabla u(x)}{\lambda}\bigg|^{r(x)}\right]\dx\leq1\right\}.
\end{align*}

It is easy to show that all the above spaces are Banach spaces, and because of \eqref{pcondition}, they are all separable and reflexive; see \cite{DHHR2011}. We also define the dual space $L^{r(\cdot)}(\Omega)^*=L^{r'(\cdot)}(\Omega)$ where the dual exponent $r'\in\mathcal{P}(\Omega)$ is defined by $\frac{1}{r(x)}+\frac{1}{r'(x)}=1.$ Regarding duality, we have the following analogue of the Riesz representation theorem in variable-exponent Lebesgue spaces, see \cite{Riesz}.
\begin{theorem}\label{Riesz}
Suppose that $1<r^-\leq r^+<\infty$. Then, for any linear functional $F\in L^{r(\cdot)}(\Omega)^*$, there exists a unique function $f\in L^{r'(\cdot)}(\Omega)$ such that
\[F(u)=\int_{\Omega}f(x)u(x) \dx\qquad\forall\, u\in L^{r(\cdot)}(\Omega).\]
\end{theorem}

Additionally, we introduce some function spaces that are frequently used in connection with mathematical models of incompressible fluids. Henceforth, $X(\Omega)^d$ will denote the space of $d$-component vector-valued functions with components from $X(\Omega)$. We also define the space of tensor-valued functions $X(\Omega)^{d\times d}$. Finally, we define the following spaces:
\begin{align*}
W^{1,r(\cdot)}_0(\Omega)&\defeq \left\{u\in W^{1,r(\cdot)}(\Omega):u=0\,\,\text{on}\,\,\partial\Omega\right\},\\
W^{1,r(\cdot)}_{0,\rm{div}}(\Omega)^d&\defeq \left\{\uu\in W^{1,r(\cdot)}_0(\Omega)^d:{\rm{div}}\,\uu=0\right\},\\
L^{r(\cdot)}_0(\Omega)&\defeq \left\{f\in L^{r(\cdot)}(\Omega):\int_{\Omega}f(x)\dx=0\right\}.
\end{align*}

Throughout the paper, we shall denote the duality pairing between $f\in X$ and $g\in X^*$ by  $\langle g,f\rangle$, and for two vectors $\boldsymbol{a}$ and $\boldsymbol{b}$, $\boldsymbol{a}\cdot \boldsymbol{b}$ denotes their scalar product; and, similarly, for two tensors $\mathbb{A}$ and $\mathbb{B}$, $\mathbb{A} \cdot \mathbb{B}$ signifies their scalar product. Also, for any Lebesgue measurable set $Q\subset\R^d$, $|Q|$ denotes the standard Lebesgue measure of the set $Q$, and $C$ signifies a generic positive constant, which may change at each appearance.

Next we define some technical tools required in this paper. First we introduce the subset $\mathcal{P}^{\rm{log}}(\Omega)\subset\mathcal{P}(\Omega)$: it will denote the set of all log-H\"older-continuous functions
defined on $\Omega$, that is the set of all functions $r$ defined on $\Omega$ such that
\begin{equation}\label{logh}
|r(x)-r(y)|\leq\frac{C_{\rm{log}}(r)}{-\log|x-y|}\qquad \forall\, x,y\in\Omega:0<|x-y|\leq\frac{1}{2}.
\end{equation}
It is obvious that classical H\"older-continuous functions on $\Omega$ automatically belong to this class. Also we define, for any $u\in L^1(\R^d)$, the Hardy--Littlewood maximal operator by
\[(Mu)(x)\defeq\sup_{r>0}\frac{1}{B_r(x)}\int_{B_r(x)}|u(y)|\dy, \qquad x \in \mathbb{R}^d,\]
where $B_r(x)$ is the open ball in $\mathbb{R}^d$ of radius $r$ centred at $x \in \mathbb{R}^d$.
Similarly, for any $\uu\in W^{1,1}(\R^d)^d$, we define $M(\nabla\uu)\defeq M(|\nabla\uu|)$.

Keeping in mind the above definition, we state the following lemma, which summarizes basic properties of Lebesgue and Sobolev spaces with a log-H\"older-continuous variable exponent. {\color{black}{For a proof, we refer to \cite{DHHR2011}, which is also an extensive source of information about variable-exponent spaces.}}

\begin{lemma}
Let $\Omega\subset\R^d$ be a bounded open Lipschitz domain and let $r\in\mathcal{P}^{\rm{log}}(\Omega)$ satisfy \eqref{pcondition}. Then the following properties hold:
\begin{itemize}
\item Density theorem, i.e.,
\[\overline{C^{\infty}(\overline{\Omega})}^{\norm{\cdot}_{1,r(\cdot)}}=W^{1,r(\cdot)}(\Omega).\]
\item Embedding theorem, i.e., if $1<r^-\leq r^+<d$ then
\[W^{1,r(\cdot)}(\Omega)\hookrightarrow L^{q(\cdot)}(\Omega)\,\,\,{\rm{provided}}\,\,\,{\rm{that}}\,\,\,1\leq q(x)\leq\frac{dr(x)}{d-r(x)}\eqdef r^*(x)\qquad\forall\, x\in\overline{\Omega}.\]
The embedding is compact whenever $q(x)<r^*(x)$ for all $x\in\overline{\Omega}$.
\item H\"older's inequality, i.e.,
\[\norm{fg}_{s(\cdot)}\leq2\norm{f}_{r(\cdot)}\norm{g}_{q(\cdot)},\,\,\,\text{with}\,\,\,r,q,s\in\mathcal{P}(\Omega),\,\,\,\frac{1}{s(x)}=\frac{1}{r(x)}+\frac{1}{q(x)}, \quad x \in \Omega.\]

\item Poincar\'e's inequality, i.e.,
\[\|u\|_{r(\cdot)}\leq C(d,C_{\rm{log}}(r))\,{\rm{diam}}(\Omega)\|\nabla u\|_{r(\cdot)}\qquad\forall\, u\in W^{1,r(\cdot)}_0(\Omega).\]

\item Korn's inequality, i.e.,
\[\norm{\nabla \uu}_{r(\cdot)}\leq C(\Omega, C_{\rm{log}}(r))\norm{\DD\uu}_{r(\cdot)}\qquad \forall\, \uu\in W^{1,r(\cdot)}_0(\Omega)^d,\]
where $C_{\rm{log}}(r)$ is the constant appearing in the definition of the class of log-H\"older-continuous functions.
\end{itemize}
\end{lemma}

Next, we recall the following generalization of McShane's extension theorem (cf. Corollary 1 in \cite{McShane}) to variable-exponent spaces and the boundedness of the maximal operator in variable-exponent context.
\begin{lemma}\label{pext}
{\rm{(Variable index extension \cite{CFN2003}})} Let $\Omega\subset\R^d$ be an bounded open  Lipschitz domain and suppose that $r\in\mathcal{P}^{\rm{log}}(\Omega)$ is arbitrary with $r^->1$. Then, there exists an extension $q\in\mathcal{P}^{\rm{log}}(\R^d)$ such that $q^-=r^-$ and $q^+=r^+$, and the Hardy--Littlewood maximal operator $M$ is bounded from $L^{q(\cdot)}(\R^d)$ to $L^{q(\cdot)}(\R^d)$.
\end{lemma}

Another relevant auxiliary result concerns the Bogovski\u{\i} operator.
{\color{black}{The following result guarantees the existence of the Bogovski\u{\i} operator in the variable-exponent setting,}} (see \cite{DHHR2011}).
\begin{theorem}\label{contibog}
Let $\Omega\subset\R^d$ be a bounded open Lipschitz domain and suppose that $r\in\mathcal{P}^{\rm{log}}(\Omega)$ with $1<r^-\leq r^+<\infty$. Then, there exists a bounded linear operator $\mathcal{B}:L^{r(\cdot)}_0(\Omega)\rightarrow W^{1,r(\cdot)}_0(\Omega)^d$ such that for all $f\in L^{r(\cdot)}_0(\Omega)$ we have
\begin{align*}
{\rm{div}}\,(\mathcal{B}f)&=f,\\
\|\mathcal{B}f\|_{1,r(\cdot)}&\leq C\|f\|_{r(\cdot)},
\end{align*}
where $C$ depends on $\Omega$, $r^-$, $r^+$, and $C_{\rm{log}}(r)$.
\end{theorem}

Using this notation, the weak formulation of the problem \eqref{eq1}--\eqref{q2}, with the nonlinear terms satisfying the assumptions above, is as follows.

{\bf{Problem (Q).}} For $\boldsymbol{f}\in (W^{1,r^-}_0(\Omega)^d)^*$, $c_d\in W^{1,q}(\Omega)$, $q>d$, and a H\"older-continuous function $r$, with $1<r^-\leq r(c)\leq r^+<\infty$ for all $c\in[c^-,c^+]$, find $(c-c_d)\in W^{1,2}_0(\Omega)\cap C^{0,\alpha}(\overline{\Omega})$, for some $\alpha \in (0,1)$, $\uu\in W^{1,r(c)}_0(\Omega)^d$, $p\in L^{r'(c)}_0(\Omega)$ such that
\begin{alignat*}{2}
\int_{\Omega}\SSS(c,\DD\uu)\cdot\nabla\boldsymbol{\psi}-(\uu\otimes \uu)\cdot\nabla\boldsymbol{\psi}\dx-\langle{\rm{div}}\,\boldsymbol{\psi},p\rangle&=\langle \boldsymbol{f},\boldsymbol{\psi}\rangle\qquad&&\forall\,\boldsymbol{\psi}\in W^{1,\infty}_0(\Omega)^d,\\
\int_{\Omega}q\,{\rm{div}}\,\uu\dx&=0\qquad &&\forall\, q\in L^{r'(c)}_0(\Omega),\\
\int_{\Omega}\q_c(c,\nabla c,\DD\uu)\cdot\nabla\varphi-c\uu\cdot\nabla\varphi\dx&=0\qquad &&\forall\,\varphi\in W^{1,2}_0(\Omega).
\end{alignat*}

Let us now state the `continuous' inf-sup condition, which has an important role in the mathematical analysis of incompressible flow problems.

\begin{proposition}
Let $\Omega\subset\R^d$ be a bounded open Lipschitz domain and $r\in\mathcal{P}^{\rm{log}}(\Omega)$ with $1<r^-\leq r^+<\infty$. Then, there exists a constant $\alpha_r>0$ such that
\[\sup_{0\neq \vv\in W^{1,r(c)}_0(\Omega)^d,\,\,\|\vv\|_{1,r(c)}\leq1} \langle{\rm{div}}\,\vv,q\rangle\geq\alpha_r\|q\|_{r'(c)}\qquad\forall\, q\in L^{r'(c)}_0(\Omega).\]
\end{proposition}

This is a direct consequence of Theorem \ref{contibog} and the following norm-conjugate formula:

\begin{lemma}
Let $r\in\mathcal{P}^{\rm{log}}(\Omega)$ be a variable exponent with $1<r^-\leq r^+<\infty$; then we have
\[\frac{1}{2}\|f\|_{r(\cdot)}\leq\sup_{g\in L^{r'(\cdot)}(\Omega),\,\,\|g\|_{r'(\cdot)}\leq1}\int_{\Omega}|f||g|\dx,\]
for all measurable functions $f \in  L^{r(\cdot)}(\Omega)$.
\end{lemma}

Thanks to the above `continuous' inf-sup condition, we can restate {\bf{Problem (Q)}} in the following (equivalent) divergence-free setting.

{\bf{Problem (P).}} For $\boldsymbol{f}\in (W^{1,r^-}_0(\Omega)^d)^*$, $c_d\in W^{1,q}(\Omega)$, $q>d$, and H\"older-continuous function $r$, with $1<r^-\leq r(c)\leq r^+<\infty$ for all $c\in[c^-,c^+]$, find $(c-c_d)\in C^{0,\alpha}(\overline{\Omega})\cap W^{1,2}_0(\Omega)$, $\uu\in W^{1,r(c)}_{0,{\rm{div}}}(\Omega)^d$, such that
\begin{align*}
\int_{\Omega}\SSS(c,\DD\uu)\cdot\nabla\boldsymbol{\psi}-(\uu\otimes \uu)\cdot\nabla\boldsymbol{\psi}\dx&=\langle \boldsymbol{f},\boldsymbol{\psi}\rangle&&\forall\,\boldsymbol{\psi}\in W^{1,\infty}_{0,{\rm{div}}}(\Omega)^d,\\
\int_{\Omega}\q_c(c,\nabla c,\DD\uu)\cdot\nabla\varphi-c\uu\cdot\nabla\varphi\dx&=0&&\forall\,\varphi\in W^{1,2}_0(\Omega).
\end{align*}

The existence of a weak solution to problem {\bf{(P)}} was proved in \cite{BP2014} in the case when the variable exponent $x \mapsto r(x)$ is bounded below by $r^->\frac{d}{2}$. As it will be made clear later, we can only perform the convergence analysis of a finite element approximations of this problem in two dimension.
\end{section}

\begin{section}{Finite element approximation}
In this section, we will construct finite element spaces, which we shall use in this paper and state the Galerkin approximation of the problem \eqref{eq1}--\eqref{q2}. The existence of a finite element solution in the discretely divergence-free setting will be established by using Brouwer's fixed point theorem. Next, we shall prove a discrete inf-sup condition to ensure the existence of a discrete pressure. Finally we will state and prove discrete counterparts of some well-known theorems, which will be key tools in the convergence analysis of the finite element approximation of the problem under consideration.

\begin{subsection}{Finite element spaces}
Let $\{\G_n\}$ be a family of shape-regular partitions of $\overline{\Omega}$ satisfying the following properties:
\begin{itemize}
\item {\bf{Affine equivalence}}: For every element $E\in \G_n$, there exists a nonsingular affine mapping
\[\boldsymbol{F}_E:E\rightarrow\hat{E},\]
where $\hat{E}$ is the standard reference $d$-simplex in $\R^d$.
\item {\bf{Shape-regularity}}: For any element $E\in \G_n$, {\color{black}{the ratio of ${\rm{diam}}\,E$ to the radius of the inscribed ball is bounded below uniformly by a positive constant, with respect to all $\G_n$ and $n\in\mathbb{N}$.}}
\end{itemize}

For a given partition $\G_n$, the finite element spaces are defined by
\begin{align*}
\V^n&=\V(\G_n)\defeq \{\VV\in C(\overline{\Omega})^d:\VV_{|E}\circ \boldsymbol{F}^{-1}_E\in\hat{\mathbb{P}}_{\V},E\in \G_n\,\,\text{and}\,\,\VV_{|\partial\Omega}=\boldsymbol{0}\},\\
\Q^n&=\Q(\G_n)\defeq \{Q\in L^{\infty}(\Omega):Q_{|E}\circ \boldsymbol{F}^{-1}_E\in\hat{\mathbb{P}}_{\Q},E\in \G_n\},\\
\Z^n&=\Z(\G_n)\defeq\{Z\in C(\overline{\Omega}):Z_{|E}\circ \boldsymbol{F}^{-1}_E\in\hat{\mathbb{P}}_{\Z},E\in \G_n\,\,\text{and}\,\,Z_{|\partial\Omega}=0\},
\end{align*}
where $\hat{\mathbb{P}}_{\V}\subset W^{1,\infty}(\hat{E})^d$, $\hat{\mathbb{P}}_{\Q}\subset L^{\infty}(\hat{E})$ and $\hat{\mathbb{P}}_{\Z}\subset W^{1,\infty}(\hat{E})$ are finite-dimensional subspaces.

$\V^n$ and $\Z^n$ are assumed to have finite and locally supported bases; for example, in the case of $\V^n$, for each $n\in\mathbb{N}$, there exists an $N_n\in\mathbb{N}$ such that
\[\V^n=\text{span}\{\VV^n_1,\ldots,\VV^n_{N_n}\}\]
and for each basis function $\VV^n_i$, $i=1,\ldots,N_n$, we have that if there exists an $E\in \G_n$ with $\VV^n_i\neq0$ on $E$, then
\[\text{supp}\,\VV^n_j\subset\bigcup\{E'\in G_n:E'\cap E\neq\emptyset\}\eqdef S_E.\]
For the pressure space $\Q^n$, we assume that $\Q^n$ has a basis consisting of discontinuous piecewise polynomials; i.e., for each $n\in\mathbb{N}$, there exists an $\tilde{N}_n\in\mathbb{N}$ such that
\[\Q^n={\rm{span}}\{Q^n_1,\ldots,Q^n_{\tilde{N}_n}\}\]
and for each basis function $Q^n_i$, we have that
\[{\rm{supp}}\,Q^n_i=E\qquad{\rm{for}}\,\,\,{\rm{some}}\,\,\,E\in\G_n.\]

We assume further that $\V^n$ contains continuous piecewise linear functions and $\Q^n$ contains piecewise constant functions.

Note further, by shape-regularity, that
\[\exists m\in\mathbb{N}:|S_E|\leq m|E|\,\,\,{\rm{for}}\,\,\,{\rm{all}}\,\,\,E\in\G_n,\]
where $m$ is independent of $n$. We denote by $h_E$ the diameter of $E$.

We also introduce the subspace $\V^n_{\rm{div}}$ of discretely divergence-free functions. More precisely, we define
\[\V^n_{\rm{div}}\defeq\{\VV\in\V^n:\langle{\rm{div}}\,\VV,Q\rangle=0\,\,\,\forall\, Q\in\Q^n\},\]
and the subspace of $\Q^n$ consisting of vanishing integral mean-value approximations:
\[\Q^n_0\defeq\{Q\in\Q^n:\int_{\Omega}Q\dx=0\}.\]

Throughout this paper, we assume that all finite element spaces introduced above have the following properties.

\textbf{Assumption 1} (Approximability): For all $s\in[1,\infty)$,
\begin{align*}
\inf_{\VV\in\V^n}\norm{\vv-\VV}_{1,s}&\rightarrow0\qquad \qquad \forall\, \vv\in W^{1,s}_0(\Omega)^d\,\,\text{as}\,\,n\rightarrow\infty, \\
\inf_{Q\in\Q^n}\norm{q-Q}_s&\rightarrow0 \qquad \qquad \forall\, q\in L^s(\Omega)\,\,\text{as}\,\,n\rightarrow\infty,\\
\inf_{Z\in\Z^n}\norm{z-Z}_{1,s}&\rightarrow0\qquad \qquad \forall\, z\in W^{1,s}_0(\Omega)\,\,\text{as}\,\,n\rightarrow\infty.
\end{align*}
For this, a necessary condition is that the maximal mesh size vanishes, i.e. we have $\max_{E\in\G_n}h_E\rightarrow0$ as $n\rightarrow\infty$.

\textbf{Assumption 2} (Existence of a projection operator $\Pi^n_{\rm{div}}$): For each $n\in\mathbb{N}$, there exists a linear projection operator $\Pi^n_{\rm{div}}:W^{1,1}_0(\Omega)^d\rightarrow\V^n$ such that:
\begin{itemize}
\item $\Pi^n_{\rm{div}}$ preserves the divergence structure in the dual of the discrete pressure space, in other words, for any $\vv\in W^{1,1}_0(\Omega)^d$, we have
\[\langle{\rm{div}}\,\vv,Q\rangle=\langle{\rm{div}}\,\Pi^n_{\rm{div}}\vv,Q\rangle\qquad\forall\, Q\in\Q^n.\]
\item $\Pi^n_{\rm{div}}$ is locally $W^{1,1}$-stable, i.e., there exists a constant $c_1>0$, independent of $n$, such that
\begin{equation}\label{3.2}
\hspace{-4mm}\fint_E|\Pi^n_{\rm{div}}\vv|+h_E|\nabla\Pi^n_{\rm{div}}\vv|\dx\leq c_1\fint_{S_E}|\vv|+h_E|\nabla \vv|\dx\qquad\forall\, \vv\in W^{1,1}_0(\Omega)^d\,\,\,{\rm{and}}\,\,\,\forall\, E\in \G_n.
\end{equation}

\end{itemize}
Note that \eqref{3.2} together with Poincar\'e's inequality implies a simpler version of local $W^{1,1}(\Omega)^d$-stability:
\begin{equation}\label{simple}
\fint_E|\nabla\Pi^n_{\rm{div}}\vv|\dx\leq c\fint_{S_E}|\nabla \vv|\dx.
\end{equation}
Note further that the local $W^{1,1}(\Omega)^d$-stability of $\Pi^n_{\rm{div}}$ implies its local and global $W^{1,s}(\Omega)^d$-stability for $s\in[1,\infty]$. Indeed, since the function $t\mapsto t^s$ {\color{black}{is a convex function for $s\in[1,\infty)$, by the equivalence of norms in finite-dimensional spaces, standard scaling arguments and Jensen's inequality, we have that}}
\begin{align*}
|\Pi^n_{\rm{div}}\vv(x)|+h_E|\nabla\Pi^n_{\rm{div}} \vv(x)|
&\leq \|\Pi^n_{\rm{div}}\vv(x)\|_{L^{\infty}(E)}+h_E\norm{\nabla\Pi^n_{\rm{div}}\vv}_{L^{\infty}(E)}\\
&\leq c\fint_E |\Pi^n_{\rm{div}}\vv(x)|+h_E|\nabla\Pi^n_{\rm{div}} \vv|\dx \\
&\leq c\fint_{S_E} |\vv(x)|+h_E|\nabla \vv|\dx\\
&\leq c\left(\fint_{S_E}|\vv(x)|^s+h^s_E|\nabla \vv|^s\dx\right)^{\frac{1}{s}}
\qquad\forall\, E\in \G_n.
\end{align*}
Raising this inequality to the $s$-th power and integrating over $E$ gives
\begin{align}\label{sta_vs}
\int_E |\Pi^n_{\rm{div}}\vv|^s+h^s_E|\nabla \Pi^n_{\rm{div}} \vv|^s\dx\leq c\int_{S_E}|\vv|^s+h^s_E|\nabla \vv|^s\dx.
\end{align}
Summing over all elements $E\in \G_n$, considering the locally finite overlap of patches (which is the consequence of the assumed shape-regularity of $\G_n$), and Poincar\'e's inequality implies, for any $s\in[1,\infty)$, that
\begin{equation}\label{sta_v}
\norm{\Pi^n_{\rm{div}} \vv}_{1,s}\leq c_s\norm{\vv}_{1,s}\qquad\forall\, \vv\in W^{1,s}_0(\Omega)^d,
\end{equation}
with a constant $c_s>0$ independent of $n>0$. With a similar argument, the same inequality can be derived for $s=\infty$. Note further that the approximability (Assumption 1) and inequality \eqref{sta_v} imply the convergence of $\Pi^n_{\rm{div}} \vv$ to $\vv$. In fact, for any $\VV\in\V^n$,
\begin{align*}
\|\vv-\Pi^n_{\rm{div}} \vv\|_{1,s}
&\leq\norm{\vv-\VV}_{1,s}+\|\VV-\Pi^n_{\rm{div}} \vv\|_{1,s}\\
&=\norm{\vv-\VV}_{1,s}+\|\Pi^n_{\rm{div}}(\vv-\VV)\|_{1,s}\\
&\leq\norm{\vv-\VV}_{1,s}+c_s\norm{\vv-\VV}_{1,s}.
\end{align*}
Hence we have
\begin{equation}\label{v_conv}
\|\vv-\Pi^n_{\rm{div}} \vv\|_{1,s}\leq C\inf_{\VV\in\V^n}\|\vv-\VV\|_{1,s}\rightarrow0\qquad\forall\, \vv\in W^{1,s}_0(\Omega)^d\,\,\text{as}\,\,n\rightarrow\infty.
\end{equation}

\textbf{Assumption 3} (Existence of a projection operator $\Pi^n_{\Q}$): For each $n\in\mathbb{N}$, there exists a linear projection operator $\Pi^n_{\Q}:L^1(\Omega)\rightarrow\Q^n$ such that $\Pi^n_{\Q}$ is locally $L^1$-stable; i.e., there exists a constant $c_2>0$, independent of $n$, such that
\begin{equation}\label{Q_L1sta}
\fint_E|\Pi^n_{\Q}q|\dx\leq c_2\fint_{S_E}|q|\dx
\end{equation}
for all $q\in L^1(\Omega)$ and all $E\in\G_n$.

Note that with the same argument as above, we have
\begin{equation}\label{Q_sta1}
\int_E|\Pi^n_{\Q}q|^{s'}\dx\leq c_{s'}\int_{S_E}|q|^{s'}\dx\qquad \forall\, E\in\G_n,\quad \forall\, q\in L^{s'}(\Omega), \quad\forall\, s'\in(1,\infty),
\end{equation}
and summing over all $E\in\G_n$ yields
\begin{equation}\label{q_sta}
\|\Pi^n_{\Q}q\|_{s'}\leq c_{s'}\|q\|_{s'}\qquad \forall\, q\in L^{s'}(\Omega), \quad\forall\, s'\in(1,\infty).
\end{equation}
Also, the stability of $\Pi^n_{\Q}$ and Assumption 1 imply that $\Pi^n_{\Q}$ satisfies
\begin{equation}\label{q_conv}
\|q-\Pi^n_{\Q}q\|_{s'}\rightarrow0,\qquad{\rm{as}}\,\,\,n\rightarrow\infty\,\,\,{\rm{for}}\,\,\,{\rm{all}}\,\,\,q\in L^{s'}(\Omega)\,\,\,{\rm{and}}\,\,\,s'\in(1,\infty).
\end{equation}

\begin{remark}
{\color{black}{According to \cite{BBDR2012}, the following pairs of velocity-pressure finite element spaces satisfy Assumptions 1, 2 and 3, for example:}}
\begin{itemize}
\item {\color{black}{The conforming Crouzeix--Raviart Stokes element, i.e., continuous piecewise quadratic plus cubic bubble velocity and discontinuous piecewise linear pressure approximation}} (compare e.g. with \cite{BF2012});
\item The space of continuous piecewise quadratic polynomials for the velocity and piecewise constant pressure approximation; see, \cite{BF2012}.
\end{itemize}
\end{remark}

Our final assumption is the existence of a projection operator for the concentration space.

\textbf{Assumption 4} (Existence of a projection operator $\Pi^n_{\Z}$): For each $n\in\mathbb{N}$, there exists a linear projection operator $\Pi^n_{\Z}:W^{1,1}_0(\Omega)\rightarrow\Z^n$ such that
\[\fint_E|\Pi^n_{\Z}z|+h_E|\nabla\Pi^n_{\Z}z|\dx\leq c_3\fint_{S_E}|z|+h_E|\nabla z|\dx\qquad\forall\, z\in W^{1,1}_0(\Omega)\,\,\,{\rm{and}}\,\,\,\forall\, E\in \G_n,\]
where $c_3$ does not depend on $n$.

Similarly as above, the projection operator $\Pi^n_{\Z}$ is globally $W^{1,s}$-stable for $s\in[1,\infty]$, and thus, by approximability,
\begin{equation}\label{z_conv}
\|\Pi^n_{\Z}z-z\|_{1,s}\rightarrow0 \qquad\forall\, z\in W^{1,s}_0(\Omega).
\end{equation}
\end{subsection}

\begin{subsection}{Stability of projection operators in variable-exponent spaces}
In this subsection, we shall state and prove some important auxiliary results regarding projection operators in the variable-exponent context. A first key step is to prove stability of the projection operator $\Pi^n_{\rm{div}}$. The main difficulty lies in the fact that we are dealing with variable-exponent spaces, so several classical results are not applicable.

To overcome this problem, we need a technical tool concerning variable-exponent spaces, which is also called the {\it{key estimate}}. We begin with a brief introduction to the {\it{key estimate}}.

In recent years, the field of variable-exponent spaces $L^{r(\cdot)}$ has been the subject of active research. A major breakthrough was the identification of the condition on the exponent $r$
which guarantees boundedness of the Hardy-Littlewood maximal operator $M$ on $L^{r(\cdot)}$: {\it{log-H\"older-continuity}}, which then enables the use of  tools from harmonic analysis.
The motivation for the {\it{key estimate}} comes from the integral version of Jensen's inequality, which states that, for every real-valued convex function $\psi$ defined on $[0,\infty)$, and every cube $Q$, we have
\[\psi\left(\fint_Q|f(y)|\dy\right)\leq\fint_Q\psi(|f(y)|)\dy.\]
{\color{black}{Therefore, we need to identify a suitable substitute for Jensen's inequality in
the context of variable-exponent spaces, which is called the {\it{key estimate}}, and is stated in the next theorem; see \cite{DS2013}.
%
%
}}

\begin{theorem}\label{keyestimate}
{\rm{(Key estimate)}}. Let $r\in\mathcal{P}^{\rm{log}}(\R^d)$ with $r^+<\infty$. Then, for every $m>0$, there exists a constant $c_1>0$, which depends only on $m$, $C_{\rm{log}}(r)$ and $r^+$, such that
\begin{equation}\label{key}
\left(\fint_Q|f(y)|\dy\right)^{r(x)}\leq c_1\fint_Q|f(y)|^{r(y)}\dy+c_1|Q|^m
\end{equation}
for every cube (or ball) $Q\subset\R^n$ with $\ell(Q)\leq1$, all $x\in Q$ and all $f\in L^1(Q)$ with
\[\fint_Q|f|\dy\leq|Q|^{-m}.\]
\end{theorem}

As a next step, we shall prove the stability of the projection operator $\Pi^n_{\rm{div}}$ in the variable-exponent context.

\begin{proposition}\label{sta_variable}
Let $r\in\mathcal{P}^{\rm{log}}(\R^d)$ with $r^+<\infty$. Then, there exists a constant $C>0$, which depends on $\Omega$, $C_{\rm{log}}(r)$ and $r^+$, such that, for all $\vv\in W^{1,r(\cdot)}_0(\Omega)^d$,
\[\int_{\Omega}|\nabla\Pi^n_{\rm{div}}\vv(x)|^{r(x)}\dx\leq C\int_{\Omega}|\nabla \vv(x)|^{r(x)}\dx+C\max_{E\in\G_n}h_E^{d+1}.\]
\end{proposition}
\begin{proof}
For $E\in\G_n$, by equivalence of norms in finite-dimensional spaces and a standard scaling argument,
\begin{align*}
\int_E|\nabla\Pi^n_{\rm{div}}\vv|^{r(x)}\dx
&\leq C\int_E\left(\fint_E|\nabla\Pi^n_{\rm{div}}\vv(y)|\dy\right)^{r(x)}\dx
\leq C\int_E\left(\fint_{S_E}|\nabla \vv(y)|\dy\right)^{r(x)}\dx\\
&\leq C\int_E\left(\fint_{S_E}|\nabla \vv(y)|^{r(y)}\dy+Ch_E^{d+1}\right)\dx\\
&\leq C\int_E\fint_{S_E}|\nabla \vv(y)|^{r(y)}\dy\dx+C\, |E| \max_{E\in\G_n}h_E^{d+1}\\
&=C\int_{S_E}|\nabla\vv(y)|^{r(y)}\dy+C\, |E| \max_{E\in\G_n}h_E^{d+1},
\end{align*}
where we have used \eqref{simple} in the second inequality and \eqref{key} in the third inequality. Summing up the above inequalities over $E\in \G_n$, we have
\[\int_\Omega|\nabla\Pi^n_{\rm{div}}\vv(x)|^{r(x)}\dx\leq C\int_{\Omega}|\nabla \vv(x)|^{r(x)}\dx+C\,|\Omega|\max_{E\in\G_n}h_E^{d+1}.\]
That completes the proof.
\end{proof}

Next, we shall investigate the stability of the projection operator $\Pi^n_{\Q}$ in variable-exponent Lebesgue spaces. To this end we shall first present some auxiliary results. The first of these is referred to as the {\it{local-to-global result}}, which is a generalization of an analogous result in classical $L^r$ spaces. In the classical Lebesgue space setting the following is easily seen to hold:
\[\|f\|_r=\left(\sum_i\|\chi_{\Omega_i}f\|^r_r\right)^{\frac{1}{r}}=\norm{\sum_i\chi_{\Omega_i}\frac{\|\chi_{\Omega_i}f\|_r}{\|\chi_{\Omega_i}\|_r}}_r,\]
where $\Omega=\cup_i\Omega_i$ and $\Omega_i\cap\Omega_j\neq\emptyset$ for $i\neq j$.
This raises the question whether in a variable-exponent space $L^{r(\cdot)}(\Omega)$ one has
\[\|f\|_{r(\cdot)}\approx\norm{\sum_i\chi_{\Omega_i}\frac{\|\chi_{\Omega_i}f\|_{r(\cdot)}}{\|\chi_{\Omega_i}\|_{r(\cdot)}}}_{r(\cdot)}.\]
This statement is indeed true, provided that $r\in\mathcal{P}^{\rm{log}}$ and $\{\Omega_i\}$ is locally $N$-finite in the sense of the following definition.
\begin{definition}
Let $N\in\mathbb{N}$. A family $\mathcal{Q}$ of measurable sets $Q\subset\R^d$ is called {\it{locally $N$-finite}} if
\[\sum_{Q\in\mathcal{Q}}\chi_Q\leq N\]
almost everywhere in $\R^d$.
\end{definition}

Let us now state the norm-equivalence theorem precisely; for its proof, see Chapter 7 in \cite{DHHR2011}.
\begin{theorem}\label{local-global}
Let $r\in\mathcal{P}^{\rm{log}}(\R^d)$ and let $\mathcal{Q}$ be a locally $N$-finite family of cubes or balls $Q\subset\R^n$. Then,
\[\norm{\sum_{Q\in\mathcal{Q}}\chi_Qf}_{r(\cdot)}\approx\norm{\sum_{Q\in\mathcal{Q}}\chi_Q\frac{\|\chi_Qf\|_{r(\cdot)}}{\|\chi_Q\|_{r(\cdot)}}}_{r(\cdot)}\]
for all $f\in L^{r(\cdot)}_{\rm{loc}}(\R^n)$. The constants, not explicitly indicated in this norm-equivalence (henceforth referred to as `implicit constants'), only depend on $C_{\rm{log}}(r)$, $d$ and $N$.
\end{theorem}

To be able to make use of the formula appearing on the right-hand side of the norm-equivalence stated in Theorem \ref{local-global}, we need to compute the variable-exponent norm $\|\chi_Q\|_{r(\cdot)}$ of the characteristic function
$\chi_Q$. Some related results are presented in Chapter 4 of \cite{DHHR2011}; what we need here is the following theorem
stated therein.
\begin{theorem}\label{charnorm}
Let $r\in\mathcal{P}^{\rm{log}}(\R^d)$. Then, for every cube or ball $Q\subset\R^d$,
\[\|\chi_Q\|_{r(\cdot)}\approx|Q|^{\frac{1}{r(x)}}\qquad{\rm{if}}\,\,\,|Q|\leq2^d\,\,\,{\rm{and}}\,\,\,x\in Q.\]
The implicit constants only depend on $C_{\rm{log}}(r)$.
\end{theorem}

Finally, we need the next lemma, which will be useful for computing a variable-exponent norm locally. To state it, we define a piecewise constant approximation of a given exponent $r(\cdot)$ by
\[r_{\rm{loc}}\defeq\sum_{E\in\G_n}r(x_E)\chi_E=\sum_{E\in\G_n}r_E\chi_E,\]
where $x_E\defeq{\rm{arg\,min}}_Er$, i.e., $r_E:=r(x_E) \leq r(x)$ for all $x \in E$. {\color{black}{What we need here is the fact that the norms $\|\cdot\|_{r(\cdot)}$ and $\|\cdot\|_{r_{\rm{loc}}(\cdot)}$ are equivalent.
To this end, we quote the following result from \cite{BBD}.}}
\begin{lemma}\label{normeq}
The norms $\|\cdot\|_{r_{\rm{loc}}(\cdot)}$ and $\|\cdot\|_{r(\cdot)}$ are equivalent on $\Q^n$.
\end{lemma}

Now we are ready to prove the stability of $\Pi^n_{\Q}$ in the variable-exponent context. The precise statement of the stability property is encapsulated in the following proposition.
\begin{proposition}\label{Q_mainsta}
For a sequence of exponents $\{r^n\}_{n\in\mathbb{N}}$, assume that $r^n\rightarrow r$ in $C^{0,\alpha}(\overline{\Omega})$ for some $\alpha \in (0,1)$. Then, there exists a constant $C$, independent of $n$, such that
\[\|\Pi^n_{\Q}q\|_{r^n(\cdot)}\leq C\|q\|_{r^n(\cdot)}\qquad \forall\, q\in L^{r^n(\cdot)}(\Omega).\]
\end{proposition}
\begin{proof}
Let $q\in L^{r^n(\cdot)}(\Omega)$. Then, by Theorem \ref{local-global} and Lemma \ref{normeq},
\begin{align*}
\|\Pi^n_{\Q}q\|_{r^n(\cdot)}
&=\norm{\sum_{E\in\G^n}\chi_E\Pi^n_{\Q}q}_{r^n(\cdot)}\\
&\leq C\norm{\sum_{E\in\G^n}\chi_E\frac{\|\chi_E\Pi^n_{\Q}q\|_{r^n(\cdot)}}{\|\chi_E\|_{r^n(\cdot)}}}_{r^n(\cdot)}
\leq C\norm{\sum_{E\in\G^n}\chi_E\frac{\|\chi_E\Pi^n_{\Q}q\|_{r^n_{\rm{loc}}(\cdot)}}{\|\chi_E\|_{r^n(\cdot)}}}_{r^n(\cdot)}.
\end{align*}
By the definition of the variable-exponent norm, one has that $\|\chi_E\Pi^n_{\Q}q\|_{r^n_{\rm{loc}}(\cdot)}\leq\|\chi_E\Pi^n_{\Q}q\|_{r^n_E}$ for each $E \in \G^n$.
Therefore, by \eqref{Q_sta1},
\[\|\Pi^n_{\Q}q\|_{r^n(\cdot)}\leq C\norm{\sum_{E\in\G^n}\chi_E\frac{\|\chi_E\Pi^n_{\Q}q\|_{r^n_E}}{\|\chi_E\|_{r^n(\cdot)}}}_{r^n(\cdot)}
\leq C\norm{\sum_{E\in\G^n}\chi_E\frac{\|\chi_{S_E}q\|_{r^n_E}}{\|\chi_E\|_{r^n(\cdot)}}}_{r^n(\cdot)}.\]
Here the constant $C$ might depend on $r^n_E$, but since $1<r^-\leq r(x)\leq r^+<\infty$, we can choose a uniform constant $C$ independent of $n$ and $E$.

At this stage, we claim that
\[\|\chi_{S_E}q\|_{r^n_E}\leq\|\chi_{S_E}q\|_{r^n_{\rm{loc}}(\cdot)}.\]
Indeed, if this were not the case, then, by the definition of the Luxembourg norm, we would have that
\[\int_{\Omega}\bigg|\frac{\chi_{S_E}q}{\|\chi_{S_E}q\|_{r^n_E}}\bigg|^{r^n_{\rm{loc}}(x)}\dx<1.\]
However, by writing $S_E=E\cup E_1\cup\cdots\cup E_j$, we have that
\[\int_{\Omega}\bigg|\frac{\chi_{S_E}q}{\|\chi_{S_E}q\|_{r^n_E}}\bigg|^{r^n_{\rm{loc}}(x)}\dx=\int_{E}\bigg|\frac{\chi_{S_E}q}{\|\chi_{S_E}q\|_{r^n_E}}\bigg|^{r^n_E}\dx+\sum_{i=1}^j\int_{E_i}\bigg|\frac{\chi_{S_E}q}{\|\chi_{S_E}q\|_{r^n_E}}\bigg|^{r^n_{\rm{loc}}(x)}\dx \geq1,\]
which is a contradiction. Hence together with Lemma \ref{normeq} again, the above claim implies that
\[\|\Pi^n_{\Q}q\|_{r^n(\cdot)}\leq C\norm{\sum_{E\in\G^n}\chi_E\frac{\|\chi_{S_E}q\|_{r^n(\cdot)}}{\|\chi_E\|_{r^n(\cdot)}}}_{r^n(\cdot)}.\]

Next we claim that
\[\|\chi_{S_E}\|_{r^n(\cdot)}\leq C\|\chi_E\|_{r^n(\cdot)}.\]
By Theorem \ref{charnorm}, for any $x \in E$,
\[\|\chi_E\|_{r^n(\cdot)}\geq C|E|^{\frac{1}{r^n(x)}}\geq C|E|^{\frac{1}{r^n_E}}\geq C|S_E|^{\frac{1}{r^n_E}}\geq C|S_E|^{\frac{1}{r^n_{S_E}}}\geq C\|\chi_{S_E}\|_{r^n(\cdot)},\]
and hence the claim is proved. Therefore, together with Theorem \ref{local-global} again, we have
\[\|\Pi^n_{\Q}q\|_{r^n(\cdot)}\leq C\norm{\sum_{E\in\G^n}\chi_{S_E}\frac{\|\chi_{S_E}q\|_{r^n(\cdot)}}{\|\chi_{S_E}\|_{r^n(\cdot)}}}_{r^n(\cdot)}
\leq C\norm{\sum_{E\in\G_n}\chi_{S_E}q}_{r^n(\cdot)}\leq C\|q\|_{r^n(\cdot)}\]
by the finite overlap property of the patches. Note that the constant $C$ above depends on $C_{\rm{log}}(r^n)$, and therefore also on $n$. However, since $r^n\rightarrow r$ in $C^{0,\alpha}(\overline{\Omega})$, this constant can be bounded uniformly by a new constant, which is independent of $n$. Thus the proof is complete.
\end{proof}

\end{subsection}

\begin{subsection}{Discrete inf-sup condition}

 The aim of this subsection is to state and prove a discrete inf-sup condition, which plays an important role in our proof of the existence of the discrete pressure and the analysis of its approximation properties. The key technical tools required in the proof of the discrete inf-sup condition are the existence of a Bogovski\u{\i} operator, stated in Theorem \ref{contibog}, and the stability property of $\Pi^n_{\rm{div}}$ shown in the previous subsection.

\begin{proposition}\label{dis_infsup}
Assume that $1<r^- \leq r^+ < \infty$ and $r^n\rightarrow r$ in $C^{0,\alpha}(\overline{\Omega})$ for some $\alpha \in (0,1)$. Then, there exists a constant $\beta>0$, independent of $n$, such that
\[\sup_{0\neq \VV\in\V^n,\,\,\|\VV\|_{1,r^n(\cdot)}\leq1} \langle{\rm{div}}\,\VV,Q\rangle\geq\frac{1}{\beta}\|Q\|_{(r^n)'(\cdot)}\qquad\forall\, Q\in \Q^n_0,\,\,\,n\in\mathbb{N}.\]
\end{proposition}

\begin{proof}
The assertion follows from the isomorphism between $(L^{r^n}_0(\Omega))^*$ and $L^{(r^n)'}_0$ (with the norm-equivalence constants bounded from above by $2$ and from below by $1/2$). In fact, it follows from Theorem \ref{contibog} that we have
\begin{align*}
\|Q\|_{(r^n)'(\cdot)}&\leq2\sup_{v\in L^{r^n(\cdot)}_0,\,\,\|v\|_{r^n(\cdot)}\leq1}\int_{\Omega} Q\,v\dx\\
&=2\sup_{v\in L^{r^n(\cdot)}_0,\,\,\|v\|_{r^n(\cdot)}\leq1}\int_{\Omega}Q\,{\rm{div}}\,(\mathcal{B}v)\dx\\
&=2\sup_{v\in L^{r^n(\cdot)}_0,\,\,\|v\|_{r^n(\cdot)}\leq1}\int_{\Omega}Q\,{\rm{div}}\,(\Pi^n_{\rm{div}}\mathcal{B}v)\dx.
\end{align*}
Now, by Theorem \ref{contibog} and Proposition \ref{sta_variable},
\[\|v\|_{r^n(\cdot)}\leq1\,\,\,\,\,{\rm{implies}}\,\,\,\,\,\|\nabla\Pi^n_{\rm{div}}\mathcal{B}v\|_{r^n(\cdot)}\leq C_1.\]
The constant $C_1$ depends on $C_{\rm{log}}(r^n)$, and therefore also on $n$. However, since $r^n\rightarrow r$ in $C^{0,\alpha}(\overline{\Omega})$, the constant $C_1$ can be bounded uniformly by a new constant, still denoted by $C_1$, which is independent of $n$. Therefore,
\begin{align*}
\|Q\|_{(r^n)'(\cdot)}
&\leq2\sup_{\|\Pi^n_{\rm{div}}\mathcal{B}v\|_{1,r^n(\cdot)}\leq C_1}\int_{\Omega}Q\,{\rm{div}}\,(\Pi^n_{\rm{div}}\mathcal{B}v)\,\dx\\
& = 2C_1\sup_{\|\Pi^n_{\rm{div}}\mathcal{B}\frac{v}{C_1}\|_{1,r^n(\cdot)}\leq 1}\int_{\Omega}Q\,{\rm{div}}\,(\Pi^n_{\rm{div}}\mathcal{B}\frac{v}{C_1})\,\dx\\
&\leq \beta\sup_{\VV\in\V^n,\,\,\|\VV\|_{1,r^n(\cdot)}\leq1}\langle{\rm{div}}\,\VV,Q\rangle.
\end{align*}
That completes the proof of the proposition.
\end{proof}

\end{subsection}

\begin{subsection}{Discrete Bogovski\u{\i} operator}
In this subsection, we construct a discrete counterpart of the Bogovski\u{\i} operator in the variable-exponent setting and explore its properties.

Suppose that $1<r^-\leq r^+<\infty$ and $r^n\rightarrow r$ in $C^{0,\alpha}(\overline{\Omega})$ for some $\alpha \in (0,1)$.
For $H\in{\rm{div}}\,\V^n$, define the linear functional $\mathcal{L}^n:L^{(r^n)'(\cdot)}(\Omega)\rightarrow\R$ by
\[\mathcal{L}^n(q)=\int_{\Omega}H\,\Pi^n_{\Q}q\dx,\qquad q\in L^{(r^n)'(\cdot)}(\Omega).\]
Then, thanks to Proposition \ref{Q_mainsta}, $\mathcal{L}^n$ is a bounded linear functional on $L^{(r^n)'(\cdot)}(\Omega)$. Hence, by Theorem \ref{Riesz}, there exists a unique $\mathcal{K}(H)\in L^{r^n(\cdot)}(\Omega)$ such that
\[\mathcal{L}^n(q)=\int_{\Omega}H\,\Pi^n_{\Q}q\dx=\int_{\Omega}\mathcal{K}(H)\,q\dx.\]
Note that since $H\in L^{r^n(\cdot)}_0(\Omega)$ and $\Pi^n_{\Q}c=c$ for all constants $c$, we have $\mathcal{K}(H)\in L^{r^n(\cdot)}_0(\Omega)$.

Now we define the discrete Bogovski\u{\i} operator. For $n\in\mathbb{N}$, we consider the linear operator $\mathcal{B}^n:{\rm{div}}\,\V^n\rightarrow\V^n$ by
\begin{equation}\label{define_dis_Bog}
\mathcal{B}^nH\defeq\Pi^n_{\rm{div}}\mathcal{B}\mathcal{K}(H)\in\V^n\qquad{\rm{for}}\,\,\,H\in{\rm{div}}\,\V^n,
\end{equation}
where $\mathcal{B}$ is defined in Theorem \ref{contibog}.

For later use, we require the following bound on $\mathcal{K}(H)$ in a variable-exponent norm:
\begin{align}
\|\mathcal{K}(H)\|_{r^n(\cdot)}
&\leq 2 \sup_{q\in L^{(r^n)'(\cdot)}(\Omega),\,\,\|q\|_{(r^n)'(\cdot)}\leq1}\int_{\Omega}\mathcal{K}(H)\,q\dx\nonumber\\
&=2 \sup_{q\in L^{(r^n)'(\cdot)}(\Omega),\,\,\|q\|_{(r^n)'(\cdot)}\leq1}\int_{\Omega}H\,\Pi^n_{\Q}q\dx\nonumber\\
&\leq C\sup_{Q\in\Q^n,\,\,\|Q\|_{(r^n)'(\cdot)}\leq1}\int_{\Omega}H\,Q\dx.
\label{Bog_est}
\end{align}

Next, we will show a relevant convergence property of the discrete Bogovski\u{\i} operator. {\color{black}{To this end, we need the following lemma, which is quoted from \cite{DKS2013}.}}

\begin{lemma}\label{pro_conv}
Let $\{\vv_n\}_{n\in\mathbb{N}}\subset W^{1,s}_0(\Omega)^d$, $s\in(1,\infty)$, such that $\vv_n\rightharpoonup \vv$ weakly in $W^{1,s}_0(\Omega)^d$ as $n\rightarrow\infty$. Then
\[\Pi^n_{\rm{div}}\vv_n\rightharpoonup \vv\qquad{\rm{weakly}}\,\,\,{\rm{in}}\,\,\,W^{1,s}_0(\Omega)^d\,\,\,{\rm{as}}\,\,\,n\rightarrow\infty.\]
\end{lemma}

Now we are ready to prove the desired convergence property of the discrete Bogovski\u{\i} operator.

\begin{proposition}\label{Bog_conv}
Suppose that $\VV^n\in\V^n$, $n\in\mathbb{N}$, and $\VV^n\rightarrow \VV$ weakly in $W^{1,s}_0(\Omega)^d$ as $n\rightarrow\infty$. Then, we have that
\[\mathcal{B}^n{\rm{div}}\,\VV^n\rightharpoonup\mathcal{B}\,{\rm{div}}\,\VV\qquad{\rm{weakly}}\,\,\,{\rm{in}}\,\,\,W^{1,s}_0(\Omega)^d\,\,\,{\rm{as}}\,\,\,n\rightarrow\infty.\]
\end{proposition}

\begin{proof}
Let us define $A^n\defeq{\rm{div}}\,\VV^n$; then, $A^n\rightharpoonup A\defeq{\rm{div}}\,\VV$ weakly in $L^s_0(\Omega)$ as $n\rightarrow\infty$. Therefore, thanks to \eqref{q_conv}, we have, for all $q\in L^{s'}(\Omega)$ by the classical Riesz representation theorem (here we shall use the same notation $\mathcal{K}$ as above, but in this case the constructed $\mathcal{K}(A^n)$
lies in a fixed-exponent space $L^s_0(\Omega)$), and since $\Pi^n_{\Q}q  \rightarrow q$ strongly in $L^{s'}(\Omega)$ by
\eqref{q_conv},  that
\[\int_{\Omega}\mathcal{K}(A^n)q\dx=\int_{\Omega}A^n\Pi^n_{\Q}q\dx\rightarrow\int_{\Omega}A\,q\dx\qquad{\rm{as}}\,\,\,n\rightarrow\infty.\]
In other words, we have that $\mathcal{K}(A^n)\rightharpoonup A$ weakly in $L^s_0(\Omega)$ as $n\rightarrow\infty$.
The Bogovski\u{\i} operator defined in Theorem \ref{contibog} is linear and continuous, and hence it is also continuous
with respect to the weak topologies of the respective spaces. Therefore, we have
$\mathcal{B}\mathcal{K}(A^n)\rightharpoonup\mathcal{B}A$ weakly in $W^{1,s}_0(\Omega)^d$ as $n\rightarrow\infty$.
Hence, by Lemma \ref{pro_conv}, $\mathcal{B}^n A^n:=\Pi^n_{\rm{div}}\mathcal{B}\mathcal{K}(A^n)\rightharpoonup\mathcal{B}A$ weakly in $W^{1,s}_0(\Omega)^d$ as $n\rightarrow\infty$.
As $A^n\defeq{\rm{div}}\,\VV^n$ and $A\defeq{\rm{div}}\,\VV$ the proof is complete. \end{proof}

\end{subsection}

\begin{subsection}{The finite element approximation}
{\color{black}{We are now ready to construct the finite element approximation of the problem}} \eqref{eq1}--\eqref{q2} and prove that the approximate problem has a solution.

An essential property of the problem \eqref{eq1}--\eqref{q2} is that, thanks to the fact that the velocity field $\uu$ is
divergence-free, the convective terms appearing in the equations are skew-symmetric. It is important to ensure
that this skew-symmetry is preserved under discretization, even though the finite element approximations to the velocity field
are now only discretely (rather than pointwise) divergence-free. We therefore
define the following trilinear forms:
\begin{align*}
B_u[\vv,\ww,\hh]&\defeq\frac{1}{2}\int_{\Omega}((\vv\otimes \hh)\cdot\nabla \ww-(\vv\otimes \ww)\cdot\nabla \hh)\dx,\\
B_c[b,\vv,z]&\defeq\frac{1}{2}\int_{\Omega}(z\vv\cdot\nabla b-b\vv\cdot\nabla z)\dx,
\end{align*}
for all $\vv,\ww,\hh\in W^{1,\infty}_0(\Omega)^d$, $b,z\in W^{1,\infty}(\Omega)$. These trilinear forms then
coincide with the trilinear forms associated with the corresponding convection terms if we are considering
pointwise divergence-free functions and also, thanks to their skew symmetry, they now also vanish
when $\ww=\hh$ and $b=z$, respectively. Explicitly, we have
\begin{align}
\begin{aligned}
B_u[\vv,\vv,\vv]&=0\,\,\,\,\,{\rm{and}}\,\,\,\,\,B_c[z,\vv,z]=0&&\forall\, \vv\in W^{1,\infty}_0(\Omega)^d,\,\,\,z\in W^{1,\infty}(\Omega),\label{Bfree}\\
B_u[\vv,\ww,\hh]&=-\int_{\Omega}(\vv\otimes \ww)\cdot\nabla \hh\dx&&\forall\, \vv,\ww,\hh\in W^{1,\infty}_{0,{\rm{div}}}(\Omega)^d,\\
B_c[b,\vv,z]&=-\int_{\Omega}b\vv\cdot\nabla z\dx&&\forall\, \vv\in W^{1,\infty}_{0,{\rm{div}}}(\Omega)^d,\,\,\,b,z\in W^{1,\infty}(\Omega).
\end{aligned}
\end{align}
Furthermore, the trilinear form $B_u[\cdot,\cdot,\cdot]$ is also bounded in a sense to be discussed below
in more detail. Observe that for $\frac{3d}{d+2}<r^-\leq r^+<d$, we have the Sobolev embedding
\[W^{1,r(\cdot)}(\Omega)^d\hookrightarrow L^{2r'(\cdot)}(\Omega)^d. \]
Then, H\"older's inequality yields that
\begin{align*}
\int_{\Omega}(\vv\otimes \ww)\cdot\nabla \hh\dx&\leq\|\vv\|_{2r'(\cdot)}\|\ww\|_{2r'(\cdot)}\|\hh\|_{1,r(\cdot)}\\
&\leq\|\vv\|_{1,r(\cdot)}\|\ww\|_{1,r(\cdot)}\|\hh\|_{1,r(\cdot)}.
\end{align*}
In the same way, we have
\[\int_{\Omega}(\vv\otimes \hh)\cdot\nabla \ww\dx\leq \|\vv\|_{1,r(\cdot)}\|\hh\|_{1,r(\cdot)}\|\ww\|_{1,r(\cdot)}.\]
Thus we obtain the bound
\begin{equation}\label{B_uest}
|B_u[\vv,\ww,\hh]|\leq\|\vv\|_{1,r(\cdot)}\|\ww\|_{1,r(\cdot)}\|\hh\|_{1,r(\cdot)}.
\end{equation}

Now, for $n\in\mathbb{N}$, we call a triple of functions $(\UU^n,P^n,C^n)\in \V^n\times\Q^n_0\times(\Z^n+c_d)$ a finite element approximation to a solution of the problem {\bf{(Q)}} if it satisfies
\begin{align}
\int_{\Omega}\SSS(C^n,\DD\UU^n)\cdot \DD \VV\dx+B_u[\UU^n,\UU^n,\VV]-\langle\text{div}\,\VV,P^n\rangle&=\langle \boldsymbol{f},\VV\rangle&&\forall\, \VV\in\V^n,\label{Q_n1}\\
\int_{\Omega}Q\,\text{div}\,\UU^n\dx&=0&&\forall\, Q\in\mathbb{Q}^n,\label{Q_n2}\\
\int_{\Omega}\q_c(C^n,\nabla C^n,\DD\UU^n)\cdot\nabla Z\dx+B_c[C^n,\UU^n,Z]&=0&&\forall\, Z\in\Z^n,\label{Q_n3}
\end{align}
where $c_d\in W^{1,q}(\Omega)$ with $q>d$ and $\boldsymbol{f}\in(W^{1,r^-}_{0}(\Omega)^d)^*$.

If we restrict the test-functions to $\V^n_{\rm{div}}$, the above problem reduces to finding $(\UU^n,C^n)\in\V^n_{\rm{div}}\times(\Z^n+c_d)$ such that
\begin{align}
\int_{\Omega}\SSS(C^n,\DD\UU^n)\cdot \DD \VV\dx+B_u[\UU^n,\UU^n,\VV]&=\langle \boldsymbol{f},\VV\rangle&&\forall\, \VV\in\V^n_{\rm{div}},\label{P_n1}\\
\int_{\Omega}\q_c(C^n,\nabla C^n,\DD\UU^n)\cdot\nabla Z\dx+B_c[C^n,\UU^n,Z]&=0&&\forall\, Z\in\Z^n.\label{P_n2}
\end{align}

The existence of a solution to the discrete problem \eqref{P_n1}, \eqref{P_n2} follows by a combination of a fixed point argument and iteration. To prove the existence of a solution, we need the following lemma, which is a consequence of Brouwer's fixed point theorem (cf. \cite{EVANS}, Chapter 9).

\begin{lemma}\label{FPT}
Suppose that $v:\R^N\rightarrow\R^N$ is a continuous function, which satisfies
\[(\exists r> 0)\;\;(\forall\, x \in \R^N\,:\, |x| = r) \qquad v(x)\cdot x\geq0.\]
Then, there exists an $x\in B(0,r)$ such that $v(x)=0$.
\end{lemma}

Let us prove the existence of a solution to the discrete problem \eqref{P_n1}, \eqref{P_n2}. Let $\{\boldsymbol{w}_i\}_{i=1}^m$, $\{z_i\}_{i=1}^\ell$ be bases of $\V^n_{\rm{div}}$ and $\Z^n$ respectively,
satisfying
$\int_{\Omega}\boldsymbol{w}_i\boldsymbol{w}_j\dx=\delta_{ij}$ for $i,j=1,\dots,m$ and $\int_{\Omega}z_iz_j\dx=\delta_{ij}$
for $i,j=1,\dots,\ell$.

We wish to find $\UU^n\in \V^n_{\rm{div}}$, $C^n\in \Z^n+c_d$ of the forms
\[\UU^n=\sum^m_{i=1}\alpha_i\boldsymbol{w}_i,\qquad C^n=\sum^{\ell}_{i=1}\beta_iz_i+c_d\]
such that
\begin{align*}
\int_{\Omega}\SSS(C^n,\DD\UU^n)\cdot \DD \VV\dx+B_u[\UU^n,\UU^n,\VV]&=\langle \boldsymbol{f},\VV\rangle&&\forall\, \VV\in\V^n_{\rm{div}},\\
\int_{\Omega}\q_c(C^n,\nabla C^n,\DD\UU^n)\cdot\nabla Z\dx+B_c[C^n,\UU^n,Z]&=0&&\forall\, Z\in\Z^n.
\end{align*}
Equivalently, we can rewrite the above as
\begin{align*}
\int_{\Omega}\SSS(C^n,\DD\UU^n)\cdot \DD \boldsymbol{w}_i\dx+B_u[\UU^n,\UU^n,\boldsymbol{w}_i]&=\langle \boldsymbol{f},\boldsymbol{w}_i\rangle,&&i=1,\dots,m,\\
\int_{\Omega}\q_c(C^n,\nabla C^n,\DD\UU^n)\cdot\nabla z_j\dx+B_c[C^n,\UU^n,z_j]&=0,&&j=1,\dots,\ell.
\end{align*}

For a given $n\in\mathbb{N}$, we will construct an iteration scheme; i.e., we will define a sequence of solutions $\{\UU^n_k,C^n_k\}$ acting on the equations iteratively. As a first step, define $C^n_1\defeq c_d\in\Z^n+c_d$, and let $\UU^n_1\in\V^n_{\rm{div}}$ be a solution of
\begin{equation}\label{first}
\int_{\Omega}\SSS(C^n_1,\DD\UU^n_1)\cdot\nabla \boldsymbol{w}_i\dx+B_u[\UU^n_1,\UU^n_1,\boldsymbol{w}_i] =\langle \boldsymbol{f},\boldsymbol{w}_i \rangle, \qquad i=1,\dots, m.
\end{equation}
The existence of such a $\UU^n_1$ can be established as follows. Define the function $A:\R^m\rightarrow\R^m$ by
\[A(\alpha_1,\ldots,\alpha_m)=(a_1(\alpha_1,\ldots,\alpha_m),\ldots,a_m(\alpha_1,\ldots,\alpha_m))\]
with
\[a_j(\alpha_1,\ldots,\alpha_m)=\int_{\Omega}\SSS(C^n_1,\DD\boldsymbol{\alpha})\cdot\nabla \boldsymbol{w}_j\dx+B_u[\boldsymbol{\alpha},\boldsymbol{\alpha},\boldsymbol{w}_i] -\langle \boldsymbol{f},\boldsymbol{w}_j \rangle,\]
where
\[\boldsymbol{\alpha}=\sum^m_{i=1}\alpha_i\boldsymbol{w}_i.\]
We note that $A$ is a continuous function on $\R^m$. Then we have, by Sobolev embedding and because the term
$B_u[\boldsymbol{\alpha},\boldsymbol{\alpha},\boldsymbol{\alpha}]$ vanishes thanks to the skew-symmetry of the
trilinear form $B_u$, that
\begin{align*}
A(\alpha_1,\ldots,\alpha_m)\cdot(\alpha_1,\ldots,\alpha_m)
&=\int_{\Omega}\SSS(C^n_1,\DD\boldsymbol{\alpha})\cdot\nabla \boldsymbol{\alpha}\dx+B_u[\boldsymbol{\alpha},\boldsymbol{\alpha},\boldsymbol{\alpha}] -\langle \boldsymbol{f},\boldsymbol{\alpha} \rangle\\
&\geq C_1\int_{\Omega}|\DD\boldsymbol{\alpha}|^{r(C^n_1)}\dx-C_2-|\langle \boldsymbol{f},\boldsymbol{\alpha}\rangle|\\
&\geq C_1\int_{\Omega}|\DD\boldsymbol{\alpha}|^{r^-}\dx-\|\boldsymbol{f}\|_{(W^{1,r^-}_0(\Omega))^*}\|\boldsymbol{\alpha}\|_{1,r^-}-C_2\\
&\geq C_1\|\boldsymbol{\alpha}\|^{r^-}_{1,r^-}-C(\varepsilon)\|\boldsymbol{f}\|_{(W^{1,r^-}_0(\Omega))^*}^{(r^-)'}-\varepsilon\|\boldsymbol{\alpha}\|^{r^-}_{1,r^-}-C_2\\
&\geq (C_1-\varepsilon)\|\boldsymbol{\alpha}\|^{r^-}_2-C_2\\
&=(C_1-\varepsilon)|(\alpha_1,\ldots,\alpha_m)|^{r^-}-C_2.
\end{align*}
Hence by Lemma \ref{FPT}, there exists an $m$-tuple $(\alpha_1,\ldots,\alpha_m)\in\R^m$ such that
$A(\alpha_1,\ldots,\alpha_m)=0$, which implies the existence of $\UU^n_1\in\V^n_{\rm{div}}$.

Now, multiplying the $i$-th equation in \eqref{first} by $\alpha_i$ and taking the sum over $i=1,\ldots,m$, we obtain
\[\int_{\Omega}\SSS(C^n_1,\DD\UU^n_1)\cdot\nabla \UU^n_1\dx+B_u[\UU^n_1,\UU^n_1,\UU^n_1] =\langle \boldsymbol{f},\UU^n_1 \rangle,\]
where the term $B_u[\UU^n_1,\UU^n_1,\UU^n_1]=0$ thanks to the skew-symmetry of the trilinear form $B_u$. Then, because of the coercivity of $\SSS$, we have
\[C_1\int_{\Omega}|\DD\UU^n_1|^{r(C^n_1)}\dx-C_2\leq\|\boldsymbol{f}\|_{(W^{1,r^-}_0(\Omega)^d)^*}\|\UU^n_1\|_{1,r^-}.\]
By Young's inequality,
\[\|\DD\UU^n_1\|^{r^-}_{r^-}\leq C(\varepsilon)\|\boldsymbol{f}\|_{(W^{1,r^-}_{0}(\Omega)^d)^*}^{(r^-)'}+\varepsilon\|\UU^n_1\|^{r^-}_{1,r^-}+C_2.\]
Finally, by Korn's inequality,
\begin{equation}\label{iteration1}
\|\UU^n_1\|_{1,r^-}\leq C,
\end{equation}
where the constant $C$ is independent of $n$.

Now, let $C^n_2\in\Z^n+c_d$ be a solution of the equation
\begin{equation}\label{second}
\int_{\Omega}\q_c(C^n_2,\nabla C^n_2,\DD\UU^n_1)\cdot\nabla z_j\dx+B_c[C^n_2,\UU^n_1,z_j] =0,\qquad j=1,\dots, \ell.
\end{equation}
As before, we define a function $B:\R^{\ell}\rightarrow\R^{\ell}$ by
\[B(\beta_1,\ldots,\beta_{\ell})=(b_1(\beta_1,\ldots,\beta_{\ell}),\ldots,b_{\ell}(\beta_1,\ldots,\beta_{\ell}))\]
with
\[b_j(\beta_1,\ldots,\beta_{\ell})=\int_{\Omega}\q_c(\beta,\nabla \beta,\DD\UU^n_1)\cdot\nabla z_j\dx+B_c[\beta,\UU^n_1,z_j],\]
where
\[\beta=\sum^{\ell}_{i=1}\beta_iz_i+c_d.\]
We note that $B$ is continuous on $\R^\ell$. Furthermore, we have that
\begin{align*}
B(\beta_1,\ldots,\beta_{\ell})\cdot(\beta_1,\ldots,\beta_{\ell})
&=\int_{\Omega}\q_c(\beta,\nabla\beta,\DD\UU^n_1)\cdot\nabla(\beta-c_d)\dx
+B_c[\beta,\UU^n_1,\beta-c_d]\\
&=:\rm{I}+\rm{II},
\end{align*}
with obvious definitions of I and II. Since $\q_c$ is linear with respect to its second variable, by \eqref{q1} and \eqref{q2} we have that
\begin{align*}
{\rm{I}}
&=\int_{\Omega}\q_c(\beta,\nabla(\beta-c_d),\DD\UU^n_1)\cdot\nabla(\beta-c_d)+\q_c(\beta,\nabla c_d,\DD\UU^n_1)\cdot\nabla(\beta-c_d)\dx\\
&\geq C_5\|\nabla(\beta-c_d)\|^2_2-C_4\int_{\Omega}|\nabla c_d||\nabla(\beta-c_d)|\dx\\
&\geq C_5\|\nabla(\beta-c_d)\|^2_2-C_4\|\nabla c_d\|_2\|\nabla(\beta-c_d)\|_2\\
&\geq C_5\|\nabla(\beta-c_d)\|^2_2-C(\varepsilon)\|\nabla c_d\|^2_2-\varepsilon\|\nabla(\beta-c_d)\|^2_2\\
&\geq (C_5-\varepsilon)\|\nabla(\beta-c_d)\|^2_2-C.
\end{align*}
Also,
\begin{align*}
{\rm{II}}&=\frac{1}{2}\int_{\Omega}(\beta-c_d)\UU^n_1\cdot\nabla c_d\dx-\int_{\Omega}c_d\UU^n_1\cdot\nabla(\beta-c_d)\dx\\
&={\rm{I'}}+{\rm{II'}},
\end{align*}
with obvious definitions of ${\rm I'}$ and ${\rm II'}$. Concerning $\rm{I'}$, for sufficiently large $t>2$ (i.e. $t \in (2,\infty)$ when $d=2$, and $t \in (2,6]$ when $d=3$), we have by Sobolev embedding that
\begin{align*}
\bigg|\int_{\Omega}(\beta-c_d)\UU^n_1\cdot\nabla c_d\dx\bigg|&\leq\|\beta-c_d\|_t\|\UU^n_1\|_{\frac{2t}{t-2}}\|\nabla c_d\|_2\\
&\leq C\|\beta-c_d\|_{1,2}\|\UU^n_1\|_{1,r^-}\\
&\leq C\|\nabla(\beta-c_d)\|_{2}\|\UU^n_1\|_{1,r^-}\\
&\leq\varepsilon\|\nabla(\beta-c_d)\|_2^2+C(\varepsilon)\|\UU^n_1\|_{1,r^-}^2.
\end{align*}
Hence,
\[{\rm{I'}}\geq-\varepsilon\|\nabla(\beta-c_d)\|_2^2-C(\varepsilon)\|\UU^n_1\|_{1,r^-}^2.\]
For the term $\rm{II'}$, since $c_d$ is bounded, we have that
\begin{align*}
\int_{\Omega}c_d\UU^n_1\cdot\nabla(\beta-c_d)\dx&\leq C\int_{\Omega}\UU^n_1\nabla(\beta-c_d)\dx\\
&\leq C\|\UU^n_1\|_2\|\nabla(\beta-c_d)\|_2\\
&\leq C(\varepsilon)\|\UU^n_1\|^2_{1,r^-}+\varepsilon\|\nabla(\beta-c_d)\|^2_2.
\end{align*}
Therefore,
\[{\rm{II}}={\rm{I'}}+{\rm{II'}}\geq-2\varepsilon\|\nabla(\beta-c_d)\|^2_2- C(\varepsilon)\|\UU^n_1\|^2_{1,r^-},\]
and by \eqref{iteration1}, and choosing $\varepsilon \in \big(0,\frac{1}{6}C_5\big)$, we have that
\[B(\beta_1,\dots,\beta_{\ell})\cdot(\beta_1,\dots,\beta_{\ell})={\rm{I}}+{\rm{II}}\geq (C_5-3\varepsilon)\|\nabla(\beta-c_d)\|_2^2-C\geq\frac{1}{2}C_5\,|(\beta_1,\dots,\beta_{\ell})|^2-C.\]
Thus, again, by Lemma \ref{FPT}, there exists an $\ell$-tuple $(\beta_1,\ldots,\beta_{\ell})\in\R^{\ell}$ such that $B(\beta_1,\ldots,\beta_{\ell})=0$, which implies the existence of $C^n_2\in\Z^n$.

Next, multiplying the $j$-th equation of \eqref{second} by $\beta_j$ and taking the sum over $j=1,\ldots,\ell$, we arrive at
\[\int_{\Omega}\q_c(C^n_2,\nabla C^n_2,\DD\UU^n_1)\cdot\nabla (C^n_2-c_d)\dx+B_c[C^n_2,\UU^n_1,C^n_2-c_d]=0.\]
Hence, by \eqref{q1} and \eqref{q2},
\begin{align*}
\|\nabla C^n_2\|^2_2
&\leq C\int_{\Omega}\q_c(C^n_2,\nabla C^n_2,\DD\UU^n_1)\cdot\nabla c_d\dx-B_c[C^n_2,\UU^n_1,C^n_2-c_d]\\
&\leq C\int_{\Omega}|\nabla C^n_2||\nabla c_d|\dx-B_c[C^n_2,\UU^n_1,C^n_2]+B_c[C^n_2,\UU^n_1,c_d]\\
&\leq\varepsilon\|\nabla C^n_2\|^2_2+C(\varepsilon)\|\nabla c_d\|^2_2+B_c[C^n_2,\UU^n_1,c_d].
\end{align*}
By integration by parts, Sobolev embedding, H\"older's inequality and Young's inequality,
\begin{align*}
B_c[C^n_2,\UU^n_1,c_d]
&=\frac{1}{2}\int_{\Omega}c_d\UU^n_1\cdot\nabla C^n_2\dx-\frac{1}{2}\int_{\Omega}C^n_2\UU^n_1\cdot\nabla c_d\dx\\
&=\frac{1}{2}\int_{\Omega}c_d\UU^n_1\cdot\nabla C^n_2\dx+\frac{1}{2}\int_{\Omega}{\rm{div}}\,(C^n_2\UU^n_1)c_d\dx\\
&=\int_{\Omega}c_d\UU^n_1\cdot\nabla C^n_2\dx+\frac{1}{2}\int_{\Omega}C^n_2({\rm{div}}\,\UU^n_1)c_d\dx\\
&\leq \|c_d\|_{\infty}\|\UU^n_1\|_2\|\nabla C^n_2\|_2+\frac{1}{2}\|c_d\|_{\infty}\|C^n_2\|_{\frac{r^-}{r^--1}}\|{\rm{div}}\,\UU^n_1\|_{r^-}\\
&\leq C\|\UU^n_1\|_{1,r^-}\|\nabla C^n_2\|_2+C\|\UU^n_1\|_{1,r^-}\|\nabla C^n_2\|_{\frac{dr^-}{(d+1)r^--d}}\\
&\leq C(\varepsilon)\|\UU^n_1\|_{1,r^-}^2+\varepsilon\|\nabla C^n_2\|_2^2,
\end{align*}
provided that $\frac{2d}{d+2}< r^{-} \leq r^{+} < d$. Therefore, by \eqref{iteration1}, we obtain
\begin{equation}\label{iteration2}
\|\nabla C^n_2\|^2_2\leq C+C\|\UU^n_1\|^2_{1,r^-}\leq C,
\end{equation}
where the constant $C$ does not depend on $n$.

Now we define $\UU^n_2\in\V^n_{\rm{div}}$ as a solution of the equation
\[\int_{\Omega}\SSS(C^n_2,\DD\UU^n_2)\cdot\nabla \boldsymbol{w}_i\dx+B_u[\UU^n_2,\UU^n_2,\boldsymbol{w}_i]=\langle \boldsymbol{f},\boldsymbol{w}_i \rangle, \qquad i=1,\dots, m,\]
and define $C^n_3\in\Z^n+c_d$ as a solution of the equation
\[\int_{\Omega}\q_c(C^n_3,\nabla C^n_3,\DD\UU^n_2)\cdot\nabla z_j\dx+B_c[C^n_3,\UU^n_2,z_j]=0,\qquad j=1,\dots, \ell.\]
The existence of such $\UU^n_2$ and $C^n_3$ can be established by the same argument as above. We continue this process so that we obtain, by iteration, a sequence of solutions
$\{\UU^n_k,C^n_k\}\in\V^n_{\rm{div}}\times\Z^n+c_d$ where $\UU^n_k$ is a solution of the equation
\begin{equation}\label{iter1}
\int_{\Omega}\SSS(C^n_k,\DD\UU^n_k)\cdot\nabla \boldsymbol{w}_i\dx+B_u[\UU^n_k,\UU^n_k,\boldsymbol{w}_i] =\langle \boldsymbol{f},\boldsymbol{w}_i \rangle, \qquad i=1,\dots, m,
\end{equation}
and $C^n_{k+1}$ is a solution of the equation
\begin{equation}\label{iter2}
\int_{\Omega}\q_c(C^n_{k+1},\nabla C^n_{k+1},\DD\UU^n_k)\cdot\nabla z_j\dx+B_c[C^n_{k+1},\UU^n_k,z_j]=0,\qquad j=1,\dots, \ell.
\end{equation}
Also, by the same argument as the one we used to derive the bounds \eqref{iteration1} and \eqref{iteration2}, we have
\[\|\UU^n_k\|_{1,r^-}\leq C_1\qquad{\rm{and}}\qquad\|\nabla C^n_k\|^2_2\leq C_2,\]
where $C_1$, $C_2$ are positive constants, independent of $k\in\mathbb{N}$.

Now we consider the spaces $\left(\V^n_{\rm{div}},\/\|\cdot\|_{1,r^-}\right)$ and $\left(\Z^n,\/\|\cdot\|_{1,2}\right)$. Both spaces are finite-dimensional, hence by the Bolzano--Weierstass theorem there exists subsequences (not relabelled) such that
\begin{alignat*}{2}
\UU^n_k&\rightarrow \UU^n \qquad &&{\rm{(strongly)}}\,\,{\rm{in}}\,\,\V^n_{\rm{div}},\\
C^n_k-c_d&\rightarrow C^n-c_d \qquad &&{\rm{(strongly)}}\,\,{\rm{in}}\,\,\Z^n.
\end{alignat*}
By the equivalence of norms in finite-dimensional spaces, as $k\rightarrow\infty$, and for each fixed $n \in \mathbb{N}$,
\begin{align*}
\|\UU^n_k-\UU^n\|_{1,\infty}&\leq C(n)\|\UU^n_k-\UU^n\|_{1,r^-}\rightarrow0,\\
\|C^n_k-C^n\|_{1,\infty}&\leq C(n)\|C^n_k-C^n\|_{1,2}\,\,\,\,\rightarrow0.
\end{align*}
Since convergence of a sequence of functions in the $L^{\infty}$-norm implies uniform convergence, we deduce that
\[
\UU^n_k\rightrightarrows \UU^n\,\,\,\,\,
\DD\UU^n_k\rightrightarrows \DD\UU^n,\,\,\,\,\,
C^n_k\rightrightarrows C^n\,\,\,\,\,{\rm{and}}\,\,\,\,\,
\nabla C^n_k\rightrightarrows \nabla C^n\,\,\,\,\,\mbox{a.e. on $\overline{\Omega}$}.
\]
Note further that since $\SSS$ and $\q_c$ are continuous, also
\[\SSS(C^n_k,\DD\UU^n_k)\rightrightarrows \SSS(C^n,\DD\UU^n)\,\,\,\,\,{\rm{and}}\,\,\,\,\,
\q_c(C^n_{k+1},\nabla C^n_{k+1},\DD\UU^n_k)\rightrightarrows \q_c(C^n,\nabla C^n,\DD\UU^n)\,\,\,\,\,\mbox{a.e. on $\overline{\Omega}$}.\]
Now if we pass to the limit $k\rightarrow\infty$ in \eqref{iter1} and \eqref{iter2}, we have
\begin{alignat*}{2}
\int_{\Omega}\SSS(C^n,\DD\UU^n)\cdot\nabla \boldsymbol{w}_i\dx+B_u[\UU^n,\UU^n,\boldsymbol{w}_i] &=\langle \boldsymbol{f},\boldsymbol{w}_i \rangle, \qquad&& i=1,\dots, m,\\
\int_{\Omega}\q_c(C^n,\nabla C^n,\DD\UU^n)\cdot\nabla z_j\dx+B_c[C^n,\UU^n,z_j] &=0,\qquad&& j=1,\dots, \ell.
\end{alignat*}
Therefore, we have established the existence of a solution to the Galerkin approximations \eqref{P_n1} and \eqref{P_n2} for a given $n>0$. The existence of a discrete solution triple for \eqref{Q_n1}--\eqref{Q_n3} then follows by the discrete inf-sup condition stated in Proposition \ref{dis_infsup}.

Our objective is now to pass to the limit $n\rightarrow\infty$. To this end we require two technical tools: a finite element
counterpart of the Acerbi--Fusco Lipschitz truncation method in variable-exponent Sobolev spaces, and a finite element counterpart of De Giorgi's regularity theorem for elliptic problems. We shall discuss these in the next two subsections, respectively. The finite element De Giorgi estimate considered here is restricted to the case of two space dimensions ($d=2$), as our proof rests on a discrete version of Meyers' regularity estimate in conjunction with Morrey's embedding
theorem, which,
by the nature of the argument, is limited to the case of $d=2$. A direct proof of a discrete De Giorgi estimate in the case of
$d\geq 2$, for Poisson's equation with a source term in $W^{-1,p}(\Omega)$ and $p>d$,
is contained in \cite{AC86}, subject to a restriction on the finite element stiffness matrix, analogous to
the assumption that is usually made to ensure that the discrete maximum principle holds. It is stated there,
without proof, that more general operators may be covered with little or no change, including, for instance, ``any uniformly elliptic operator in divergence form with bounded measurable coefficients''. Indeed, Casado-D\'\i az et al. \cite{CCGGM07} consider linear elliptic problems of the form $- \mbox{div}(A \nabla u) = f$ with $A \in L^\infty(\Omega)^{d\times d}$
 uniformly elliptic and $f\in L^1(\Omega)$, and assume diagonal dominance of the associated finite element
stiffness matrix, a condition, which now also involves the bounded measurable matrix function $A$ (cf. (1.17) there). As in our setting the concentration equation is nonlinear,
and the diffusion coefficient is a nonlinear function of both the concentration and the Frobenius
norm of the velocity gradient, it is unclear how exactly such a diagonal dominance condition on the associated stiffness matrix would translate into a practically verifiable restriction on the sequence of triangulations. We have therefore
confined ourselves here to the case of $d=2$. 
\end{subsection}

\begin{subsection}{Discrete Lipschitz truncation}

{\color{black}{The Lipschitz truncation method has a crucial role in the proof of our main result,}} which will be stated in the next section. In this section, we shall introduce a discrete Lipschitz truncation, acting on finite element spaces, following the ideas by Diening et al. in \cite{DKS2013}, as the composition of a `continuous' Lipschitz truncation and the projection defined in Assumption 2. For this reason, as a starting point for the construction, we shall first recall
a  result by  Diening et al. \cite{DMS2008}  concerning Lipschitz truncation in $W^{1,1}_0(\Omega)^d$, which refines the original estimates by Acerbi \& Fusco \cite{AF1988}. Note that in the following theorem the no-slip boundary condition on $\partial\Omega$ is preserved under Lipschitz truncation.

Let $\vv\in W^{1,1}_0(\Omega)^d$. We can then assume that $\vv\in W^{1,1}(\R^d)^d$ by extending $\vv$ by zero outside $\Omega$. For fixed $\lambda>0$, we define
\[\U_{\lambda}(\vv)\defeq\{M(\nabla \vv)>\lambda\}\]
and
\[\mathcal{H}_{\lambda}(\vv)\defeq\R^d\setminus(\U_{\lambda}(\vv)\cap\Omega)=\{M(\nabla \vv)\leq\lambda\}\cup(\R^d\setminus\Omega).\]
As $M(\nabla \vv)$ is lower-semicontinuous, the set $\U_{\lambda}(\vv)$ is open and the set $\mathcal{H}_{\lambda}(\vv)$ is closed.

\begin{theorem}\label{CLT}
Let $\lambda>0$ and $\vv\in W^{1,1}_0(\Omega)^d$. Then there exists a Lipschitz truncation $\vv_{\lambda}\in W^{1,\infty}_0(\Omega)^d$ satisfying the following properties:
\begin{itemize}
\item[(a)] $\vv_{\lambda}=\vv$ on $\mathcal{H}_{\lambda}(\vv)$, i.e., $\{\vv\neq \vv_{\lambda}\}\subset\{M(\nabla \vv)>\lambda\}\cap\Omega$;
\item[(b)] $\|\vv_{\lambda}\|_s\leq C\|\vv\|_s$ for all $s\in[1,\infty]$, provided that $\vv \in L^s(\Omega)^d$;
\item[(c)] $\|\nabla \vv_{\lambda}\|_s\leq C\|\nabla \vv\|_s$ for all $s\in[1,\infty]$, provided that $\vv \in W^{1,s}(\Omega)^d$;
\item[(d)] $\|\nabla \vv_{\lambda}\|_{\infty}\leq C\lambda$ almost everywhere in $\R^d$.
\end{itemize}
The constant $C$ in the inequalities stated in parts (b), (c) and (d) depends on $\Omega$ and $d$. In (b) and (c),
the constant $C$ additionally depends on $s$.
\end{theorem}

{\color{black}{Next, following Diening et al. \cite{DKS2013}, we modify the `continuous' Lipschitz truncation so that the resulting truncation is again a finite element function.

Since $\V^n\subset W^{1,1}_0(\Omega)^d$ for all $n\in\mathbb{N}$,  we can apply Theorem \ref{CLT} with arbitrary $\lambda>0$. Note however that the Lipschitz truncation $\VV_{\lambda}$ of $\VV\in\V^n$ is not contained in $\V^n$  in general. Thus we define the discrete Lipschitz truncation by}}
\begin{equation}\label{DLTdef}
\VV_{\lambda}^n\defeq\Pi^n_{\rm{div}}\circ \VV_{\lambda}\in\V^n.
\end{equation}

According to the next lemma, which we quote from \cite{DKS2013} (cf. Lemma 14 in \cite{DKS2013}), the projection operator $\Pi^n_{\rm{div}}$ modifies $\VV_{\lambda}$ in a neighborhood of $\U_{\lambda}(\VV)$ only.

\begin{lemma}\label{DCTL1}
Let $\VV\in\V^n$; then, we have that
\[\{\VV_{\lambda}^n\neq \VV\}\subset\Omega_{\lambda}^n(\VV)\defeq{\rm{interior}}\left(\bigcup\{ S_E:E\in \G_n\,\,{\rm{with}}\,\,E\cap\mathcal{U}_{\lambda}(\VV)\neq\emptyset\}\right).\]
\end{lemma}

The set $\Omega^n_{\lambda}(\vv)$ from Lemma \ref{DCTL1} is clearly larger than $\U_{\lambda}(\VV)\cap\Omega$. {\color{black}{However, according to the following result, we can still control the increase of the set. This is the most important step in the construction of the discrete Lipschitz truncation;}}
Lemma \ref{DLTL2} is,
again, quoted from \cite{DKS2013}.

\begin{lemma}\label{DLTL2}
For $n\in\mathbb{N}$, $\VV\in\V^n$ and $\lambda>0$, let $\Omega^n_{\lambda}(\VV)$ be defined as in {\rm{Lemma \ref{DCTL1}}}. Then, there exists a constant $\kappa\in(0,1)$, only depending on $\hat{\mathbb{P}}_{\V}$ and the shape-regularity of $\G_n$, such that
\[\U_{\lambda}(\VV)\cap\Omega\subset\Omega^n_{\lambda}(\VV)\subset\U_{\kappa\lambda}(\VV)\cap\Omega.\]
\end{lemma}
{%
\color{blue}{
}
}

Now we are ready to state and prove the discrete Lipschitz truncation theorem, which has a suitable form for our problem. The couple $(\VV^n,C^n)$ denotes a sequence of approximate solutions and define the variable Lebesgue exponent
\[r^n(x)\defeq(r\circ C^n)(x)\qquad{\rm{for}}\,\,{\rm{all}}\,\,x\in\overline{\Omega}.\]
{\color{black}{The following theorem is a generalization of  the result stated in Theorem \ref{CLT}. Here, however, we have the added difficulty that the variable exponent changes with the given sequence.}}

\begin{theorem}\label{LTmain}
Let $\Omega\subset\R^d$ be an bounded open Lipschitz domain and suppose that $\{\VV^n,r^n\}$ is a sequence satisfying $1<r^-\leq r^n(x)\leq r^+<\infty$ for all $x\in\overline{\Omega}$ and
\begin{align}
\VV^n&\rightharpoonup \VV\qquad{\rm{weakly}}\,\,{\rm{in}}\,\,{W^{1,r^-}_0(\Omega)^d},\label{LTv}\\
r^n&\rightarrow r\qquad \,\,\,{\rm{strongly}}\,\,{\rm{in}}\,\, C^{0,\alpha}(\overline{\Omega})\label{LTp}
\end{align}
for some $\alpha\in(0,1)$. Assume further that, for all $n\in\mathbb{N}$,
\begin{equation}\label{main4.4}
\int_{\Omega}|\nabla \VV^n|^{r^n(x)}\dx\leq C.
\end{equation}
 Then, for each $j\in\mathbb{N}$, there exists a sequence $\{\lambda_j^n\}_{n \in \mathbb{N}}$ such that
\begin{equation}\label{main4.6}
(2^j)^{2^j}\leq\lambda_j^n<(2^{j+1})^{2^{j+1}},
\end{equation}
and a sequence of Lipschitz truncations $\{\VV_j^n\}_{n \in \mathbb{N}}\subset \V^n\subset W^{1,\infty}(\Omega)^d$ such that
for all $n,j\in\mathbb{N}$,
\begin{equation}\label{main4.7}
\|\nabla \VV_j^n\|_{\infty}\leq C\lambda_j^n\leq C(2^{j+1})^{2^{j+1}}.
\end{equation}
In addition, we can extract a (not relabelled) subsequence with respect to $n$ such that, for each $j\in\mathbb{N}$,
\begin{align}
\VV_j^n&\rightarrow \VV_{\!j}&&{\rm{strongly}}\,\,{\rm{in}}\,\,L^{\sigma}(\Omega)^d\,\,{\rm{for}}\,\,{\rm{all}}\,\,\sigma\in(1,\infty),\label{main4.8}\\
\VV_j^n&\rightharpoonup \VV_{\!j}&&{\rm{weakly}}\,\,{\rm{in}}\,\,W^{1,\sigma}(\Omega)^d\,\,{\rm{for}}\,\,{\rm{all}}\,\,\sigma\in(1,\infty),\label{main4.9}\\
\nabla \VV^n_j&\rightharpoonup^*\nabla \VV_{\!j}&&{\mbox{weakly$^*$}}\,\,{\rm{in}}\,\,L^{\infty}(\Omega)^{d\times d},\label{main4.10}
\end{align}
where $\VV_{\!j}\in W^{1,\infty}(\Omega)^d$. Moreover,
\begin{equation}\label{main4.11}
\|\nabla \VV_{\!j}\|_{r(\cdot)}\leq C,
\end{equation}
and we can extract a (not relabelled) subsequence so that
\begin{equation}\label{vjconv}
\VV_{\!j}\rightharpoonup \VV\qquad{\rm{weakly}}\,\,{\rm{in}}\,\,W^{1,r(\cdot)}(\Omega)^d.
\end{equation}

Furthermore, if we extend $\VV^n$ outside $\overline\Omega$ by zero, we have
\begin{equation}\label{main4.12}
\{x\in\Omega:\VV^n_j\neq \VV^n\}\subset\{x\in\Omega:M(\nabla \VV^n)>\kappa\lambda^n_j\},
\end{equation}
where $\kappa$ is defined in {\rm{Lemma \ref{DLTL2}}}, and for all $n,j$,
\begin{equation}\label{main4.13}
\int_{\Omega}|\nabla \VV^n_j\chi_{\{\VV^n_j\neq \VV^n\}}|^{r^n(x)}\dx\leq C\int_{\Omega}|\lambda^n_j\chi_{\{\VV^n_j\neq \VV^n\}}|^{r^n(x)}\dx\leq\frac{C}{2^j}.
\end{equation}
\end{theorem}

\begin{proof}
We first extend each $\VV^n$ outside $\overline{\Omega}$ by zero and we extend each $r^n$ defined as in Lemma \ref{pext}. Then we have
\begin{alignat*}{2}
\VV^n&\rightharpoonup \VV&&\qquad{\rm{weakly}}\,\,{\rm{in}}\,\,{W^{1,r^-}(\R^d)^d},\\
r^n&\rightarrow r&&\qquad{\rm{strongly}}\,\,{\rm{in}}\,\,C^{0,\alpha}(\R^d).
\end{alignat*}
By boundedness of the maximal operator for $r^n(x)>1$, we have that
\[\|M(\nabla \VV^n)\|_{r^n(\cdot)}\leq C(n)\|\nabla \VV^n\|_{r^n(\cdot)}.\]
 Note that the constant $C(n)$ depends on $C_{\rm{log}}(r^n)$, but by the assumption $r^n\rightarrow r$ in $C^{0,\alpha}(\overline{\Omega})$, $C(n)$ can be bounded by some uniform constant $C$ independent of $n\in\mathbb{N}$. Thus directly from \eqref{main4.4}, we have
\begin{equation}\label{main4.14}
\int_{\R^d}|M(\nabla \VV^n)|^{r^n(x)}\dx\leq C.
\end{equation}
Now, for each $j\in\mathbb{N}$, define the sequence $\{\theta^i_j\}^{2^{j+1}-1}_{i=2^j}$ by
\[\theta^i_j\defeq(2^j)^i,\]
and a sequence of subsets $\{U^i_{j,n}\}^{2^{j+1}-1}_{i=2^j}$ as
\[U^i_{j,n}\defeq\{x\in\R^d:\kappa\theta^i_j<M(\nabla \VV^n)(x)\leq\kappa\theta^{i+1}_j\}.\]
Note that $U^i_{j,n}$ are mutually disjoint bounded sets, and thus
\[\sum_{i=2^j}^{2^{j+1}-1}\int_{U^i_{j,n}}|M(\nabla \VV^n)|^{r^n(x)}\dx\leq\int_{\R^d}|M(\nabla \VV^n)|^{r^n(x)}\dx\leq C.\]
By the pigeon hole principle, there exists an $i^*\in\{2^j,\ldots,2^{j+1}-1\}$ such that
\[\int_{U^{i^*}_{j,n}}|M(\nabla \VV^n)|^{r^n(x)}\dx\leq\frac{C}{2^j}.\]
Then, for this $i^*$, we set
\[\lambda^n_j\defeq\theta_j^{i^*}=(2^j)^{i^*},\]
and thus \eqref{main4.6} follows. Therefore we have
\begin{equation}\label{main4.15}
\int_{\{\kappa\lambda^n_j<M(\nabla \VV^n)\leq\kappa2^j\lambda^n_j\}}|M(\nabla \VV^n)|^{r^n(x)}\dx\leq\frac{C}{2^j}.
\end{equation}
Having such a $\lambda^n_j$, we can use \eqref{DLTdef} with $\lambda=\lambda^n_j$ applied to $\VV^n$ and thus we introduce
\[\VV^n_j\defeq \VV_{\lambda_j^n}^n.\]
Then, by Theorem \ref{CLT}, part (d), and the $W^{1,\infty}(\Omega)^d$-stability of $\Pi^n_{\rm{div}}$, we have \eqref{main4.7}. Additionally, combining Lemma \ref{DCTL1} and Lemma \ref{DLTL2} yields \eqref{main4.12}. To prove \eqref{main4.13}, we use \eqref{main4.7} and \eqref{main4.15}, and thus
\begin{align*}
\int_{\{\VV^n_j\neq \VV^n\}}|\nabla \VV^n_j|^{r^n(x)}\dx&\leq C\int_{\{\VV^n_j\neq \VV^n\}}|\kappa\lambda^n_j|^{r^n(x)}\dx\leq C\int_{\{\kappa\lambda^n_j<M(\nabla \VV^n)\}}|\kappa\lambda^n_j|^{r^n(x)}\dx\\
&=C\int_{U^{i^*}_{j,n}}|\kappa\lambda^n_j|^{r^n(x)}\dx+C\int_{\{\kappa2^j\lambda^n_j<M(\nabla \VV^n)\}}|\kappa\lambda^n_j|^{r^n(x)}\dx\\
&\leq C\int_{U^{i^*}_{j,n}}(M(\nabla \VV^n))^{r^n(x)}\dx+C\int_{\R^d}\left(\frac{M(\nabla \VV^n)}{2^j}\right)^{r^n(x)}\dx\\
&\leq\frac{C}{2^j}+\frac{C}{(2^j)^{r^-}}\int_{\R^d}(M(\nabla \VV^n))^{r^n(x)}\dx\leq\frac{C}{2^j}.
\end{align*}
By compact embedding, \eqref{main4.7}, and the fact that $\VV^n_j$ are compactly supported in $\R^d$, we can, for arbitrarily fixed $j\in\mathbb{N}$, extract a subsequence satisfying \eqref{main4.8}--\eqref{main4.10}. {\color{black}{Furthermore, by using a diagonal process, we can extract a further subsequence in $n$ such that \eqref{main4.8}--\eqref{main4.10} hold for each $j\in\mathbb{N}$.}} Finally, from \eqref{LTv}, \eqref{main4.8}, \eqref{main4.12} and H\"older's inequality, we obtain
\begin{align*}
\|\VV_j-\VV\|_1&\leq\lim_{n\rightarrow\infty}\int_{\Omega}|\VV_j-\VV^n_j|\dx+\lim_{n\rightarrow\infty}\int_{\Omega}|\VV^n_j-\VV^n|\dx+\lim_{n\rightarrow\infty}\int_{\Omega}|\VV^n-\VV|\dx\\
&=\lim_{n\rightarrow\infty}\int_{\Omega}|\VV^n_j-\VV^n|\dx\leq C\limsup_{n\rightarrow\infty}|\{\VV^n_j\neq \VV^n\}|^{\frac{1}{(r^-)'}}\\
&\leq C\limsup_{n\rightarrow\infty}|\{M(\nabla \VV^n)>\kappa\lambda^n_j\}|^{\frac{1}{(r^-)'}}\leq C\limsup_{n\rightarrow\infty}\left(\int_{\Omega}\frac{M(\nabla \VV^n)}{\kappa\lambda^n_j}\dx\right)^{\frac{1}{(r^-)'}}\\
&\leq\limsup_{n\rightarrow\infty}\frac{C}{(\lambda^n_j)^{\frac{1}{(r^-)'}}}\leq\frac{C}{(2^j)^{\frac{2^j}{(r^-)'}}}\leq \frac{C}{2^j}\qquad{\rm{for}}\,\,{\rm{sufficiently}}\,\,{\rm{large}}\,\,j\in\mathbb{N}.
\end{align*}
Consequently, we have that for a (not relabelled) subsequence, $\VV_j\rightarrow \VV$ a.e. in $\Omega$ as $j\rightarrow\infty$. So if we prove \eqref{main4.11}, by the uniqueness of the weak limit, \eqref{vjconv} follows. To prove \eqref{main4.11}, we note that
\begin{align*}
\liminf_{n\rightarrow\infty}\int_{\Omega}|\nabla \VV^n_j|^{r^n(x)}\dx
&=\liminf_{n\rightarrow\infty}\int_{\{\VV^n_j=\VV^n\}}|\nabla \VV^n|^{r^n(x)}\dx+\liminf_{n\rightarrow\infty}\int_{\{\VV^n_j\neq \VV^n\}}|\nabla \VV^n_j|^{r^n(x)}\dx\leq C
\end{align*}
which, by weak lower-semicontinuity (for the details, see the argument leading to \eqref{maininclu}) implies the bound
\[\int_{\Omega}|\nabla \VV_j|^{r(x)}\dx\leq C.\]
That completes the proof of the theorem. \end{proof}
\end{subsection}

\begin{subsection}{Uniform H\"older norm bound in two space dimensions}

When studying numerical approximations to nonlinear partial differential equations, it is often the case that, in order to prove convergence of the sequence of numerical approximations to a solution of the original problem, some \textit{a priori} knowledge about the regularity of the discrete solution is helpful. The aim of this section is to summarize some results of this type, whose continuous counterparts are well-known in the context of PDE analysis thanks to, primarily, the work of De Giorgi, Nash and Moser,
and which will be required here in order to complete the convergence analysis of the numerical method under consideration. In \cite{BS2008}, the authors formulate a Meyers type regularity estimate for the sequence of approximate solutions to a second-order linear elliptic equation obtained by a finite element method. As a corollary, by Morrey's embedding theorem, in two space dimensions at least, we will obtain a uniform bound on
a H\"older norm of the sequence of approximate solutions. We shall discuss the approximation scheme and the associated discrete De Giorgi theorem in more detail.

From the definition of the finite element space we have constructed, we know that $\Z^n\subset W^{1,\infty}_0(\Omega)$. So we can consider a conforming finite element approximation from $\Z^n$ to the weak solution $c \in W^{1,2}_0(\Omega)$ of the problem
$- \nabla \cdot (A \nabla c) = \nabla \cdot \boldsymbol{F} + h$, for $\boldsymbol{F}
\in L^p(\Omega)^d$, $h \in L^{\frac{dp}{d+p}}(\Omega)$, $p>d$, and $A\in L^\infty(\Omega)^{d\times d}$ uniformly elliptic, with the approximation $W^n \in \Z^n$
defined by:
\begin{equation}\label{M1.4}
\int_{\Omega}A(x)\nabla W^n(x)\cdot\nabla Z^n(x)\dx= -\int_\Omega \boldsymbol{F} \cdot \nabla Z^n\dx + \int_\Omega h(x) Z^n(x)\dx\qquad\forall\, Z^n\in\Z^n.
\end{equation}
An application of the Lax--Milgram theorem implies the existence of a unique solution to equation \eqref{M1.4}. {\color{black}{Moreover, as a direct consequence of Proposition 8.6.2 in \cite{BS2008} and Theorem 5.1 in \cite{Grisvald}, we have the following result.}}

\begin{theorem}\label{discreteDe}
Assume that $\Omega\subset\R^d$, $d \in \{2,3\}$, is a bounded open convex polytopal domain and
$A\in L^{\infty}(\Omega)^{d \times d}$ is uniformly elliptic. Then, there exist constants $C>0$, $n_0 \geq 1$ and $\varepsilon>0$, such that, for all $n \geq n_0$, $p\in(2,2+\varepsilon)$ and all $\boldsymbol{F}\in L^{p}(\Omega)^d$, the solution $W^n\in\Z^n$ of \eqref{M1.4} satisfies
\[\|W^n\|_{W^{1,p}(\Omega)}\leq C\left(\|\boldsymbol{F}\|_{L^{p}(\Omega)} + \|h\|_{L^{\frac{dp}{d+p}}(\Omega)}\right).\]
In particular, if $d=2$, by Morrey's embedding theorem,
we have
\[\|W^n\|_{C^{0,\alpha}(\overline{\Omega})} \leq C\left(\|\boldsymbol{F}\|_{L^{p}(\Omega)} + \|h\|_{L^{\frac{dp}{d+p}}(\Omega)}\right)\quad\mbox{with $\alpha= 1 - \frac{2}{p}\in(0,1)$}.\]
\end{theorem}

Since we need the second inequality stated in the above theorem in the subsequent analysis, we shall henceforth restrict ourselves to the case of $d=2$, and will assume that $\Omega$ is a bounded open convex polygonal domain in $\R^2$.  Obtaining a De Giorgi type regularity result for the sequence of finite element approximations to \eqref{M1.4} is a challenging open problem in the case of $d=3$.

Once we have the above result, by standard boundary reduction argument, we can obtain a similar result for the equation \eqref{M1.4} with nonhomogeneous Dirichlet boundary datum $c_d$. Indeed, if we consider $Y^n\defeq W^n-c_d$ instead of $W^n$, we have the following definition of the approximate solution
\[\int_{\Omega}A(x)\,\nabla Y^n(x)\cdot\nabla Z^n(x)\dx= -\int_\Omega \boldsymbol{F} \cdot \nabla Z^n\dx
+ \int_\Omega h(x) Z^n(x)\dx
-\int_{\Omega}A(x)\,\nabla c_d(x)\cdot\nabla Z^n(x)\dx\]
for all $Z^n \in \mathbb{Z}^n$.
We choose $q$ such that $d=2<p\leq q<2+\varepsilon$ where $\varepsilon$ is as in Theorem \ref{discreteDe}. Then, if $\boldsymbol{F}\in L^{p}(\Omega)^d$ and $c_d\in W^{1,q}(\Omega)$, it is easy to show that
$\boldsymbol{G}:=\boldsymbol{F}+A\nabla c_d \in L^p(\Omega)^d$ again. Therefore, we have the following corollary, which will be used in the subsequent analysis.

\begin{corollary}\label{disDe}
Assume that $\Omega\subset\R^2$ is a bounded open convex polygonal domain and that $A\in L^{\infty}(\Omega)^{2 \times 2}$ is uniformly elliptic with ellipticity constant $\lambda>0$. Then, there exists a $q>2$ such that the following holds: for any $\boldsymbol{G}\in L^q(\Omega)^2$, $h\in L^{\frac{2q}{q+2}}(\Omega)$ and any $c_d\in W^{1,q}(\Omega)$, there exists a unique $W^n\in\Z^n+c_d$ such that $W^n-c_d\in\Z^n\cap C^{0,\alpha}(\overline{\Omega})$ for some $\alpha\in(0,1)$, satisfying
\[\int_{\Omega}A(x) \nabla W^n(x)\cdot\nabla Z^n(x)\dx=-\int_{\Omega}\boldsymbol{G}(x)\cdot\nabla Z^n(x)\dx+\int_{\Omega}h\,Z^n\dx\qquad\forall\, Z^n\in\Z^n,\]
and fulfilling the uniform bound
\[\|W^n\|_{W^{1,q}(\Omega)\cap C^{0,\alpha}(\overline{\Omega})}\leq C\left(\Omega,\lambda,q,\|A\|_{\infty},\|\boldsymbol{G}\|_q,\|h\|_{\frac{2q}{q+2}},\|c_d\|_{1,q}\right).\]
\end{corollary}

\end{subsection}

\end{section}

\begin{section}{The main theorem}
We are now ready to state and prove our main theorem. Note that because of the restriction $d=2$ in Corollary \ref{disDe}, we only consider a two-dimensional convex polygonal domain $\Omega$. Also, we need a stronger condition on $r(x)$.
\begin{theorem}\label{mainthm}
Assume that $\Omega\subset\R^2$ is a convex polygonal domain,
and $c_d\in W^{1,q}(\Omega)$ for some $q>2$. Let us assume that $r:\R_{\geq0}\rightarrow\R_{\geq0}$ is a H\"older-continuous function with $\frac{3}{2}<r^-\leq r(c)\leq r^+<2$ for all $c\in[c^-,c^+]$ and let $\boldsymbol{f}\in(W^{1,r^-}_{0}(\Omega)^2)^*$.
Let $\{\V^n,\Q^n,\Z^n\}_{n\in\mathbb{N}}$ be the sequence of finite element space triples from {\rm{Section 4.1}} and let $\{\UU^n,P^n,C^n\}_{n\in\mathbb{N}}$ be a sequence of discrete solution triples defined by the finite element approximation \eqref{Q_n1}--\eqref{Q_n3}. Then, there exists a (not relabelled) subsequence $\{\UU^n,P^n,C^n\}_{n \in \mathbb{N}}$, which converges to a weak solution $\{\uu,p,c\}$ of \eqref{eq1}--\eqref{eq3} defined in {\bf{Problem (Q)}} as $n\in\mathbb{N}$ tends to $\infty$ in the following sense:
\begin{align*}
\UU^n&\rightharpoonup \uu \qquad{\rm{weakly}}\,\,{\rm{in}}\,\, W^{1,r^-}_{0}(\Omega)^2,\\
P^n&\rightharpoonup p \qquad\,{\rm{weakly}}\,\,{\rm{in}}\,\,L^{(r^+)'}_0(\Omega),\\
C^n&\rightharpoonup c \qquad\,{\rm{weakly}}\,\,{\rm{in}}\,\,W^{1,2}(\Omega),\\
C^n&\rightarrow c \qquad\,{\rm{strongly}}\,\,{\rm{in}}\,\,C^{0,\alpha}(\overline{\Omega})\,\,\,\rm{for}\,\,\,\rm{some}\,\,\,\alpha\in(0,1).
\end{align*}
\end{theorem}

\begin{subsection}{Convergence of the finite element approximations}
{\color{black}{As a first step in the proof of our main theorem, we pass to the limit in the sequence of solution triples and show the existence of a weak limits for the sequences in question.}}
 First, we test with $\UU^n$ in \eqref{Q_n1} and then thanks to \eqref{Q_n2} and \eqref{Bfree}, we have
 \[\int_{\Omega}\SSS(C^n,\DD\UU^n)\cdot \DD\UU^n\dx=\langle \boldsymbol{f},\UU^n\rangle.\]
 By using \eqref{S3}, duality estimates, Young's inequality and Korn's inequality, we obtain
\begin{equation}\label{UE1}
\int_{\Omega}|\nabla \UU^n|^{r(C^n)}+|\SSS(C^n,\DD\UU^n)|^{r'(C^n)}\dx\leq C_1,
\end{equation}
where $C_1$ is independent of $n$.

 Next, we test with $C^n-c_d$ in \eqref{Q_n3}, and note that by \eqref{Bfree} we have
\[\int_{\Omega}\q_c(C^n,\nabla C^n,\DD\UU^n)\cdot\nabla C^n\, \dx=\int_{\Omega}\q_c(C^n,\nabla C^n,\DD\UU^n)\cdot\nabla c_d\dx+B_c[C^n,\UU^n,c_d].\]
By \eqref{q1}, \eqref{q2}, H\"older's inequality and Young's inequality,
\begin{align*}
\|\nabla C^n\|_2^2&\leq C\int_{\Omega}|\nabla C^n|\,|\nabla c_d|\dx+B_c[C^n,\UU^n,c_d]\\
&\leq\varepsilon\|\nabla C^n\|^2_2+C(\varepsilon)\|\nabla c_d\|^2_2+B_c[C^n,\UU^n,c_d].
\end{align*}
By integration by parts, Sobolev embedding, H\"older's inequality and Young's inequality,
\begin{align*}
B_c[C^n,\UU^n,c_d]
&=\frac{1}{2}\int_{\Omega}c_d\UU^n\cdot\nabla C^n\dx-\frac{1}{2}\int_{\Omega}C^n\UU^n\cdot\nabla c_d\dx\\
&=\frac{1}{2}\int_{\Omega}c_d\UU^n\cdot\nabla C^n\dx+\frac{1}{2}\int_{\Omega}{\rm{div}}\,(C^n\UU^n)c_d\dx\\
&=\int_{\Omega}c_d\UU^n\cdot\nabla C^n\dx+\frac{1}{2}\int_{\Omega}C^n({\rm{div}}\,\UU^n)c_d\dx\\
&\leq \|c_d\|_{\infty}\|\UU^n\|_2\|\nabla C^n\|_2+\frac{\|c_d\|_{\infty}}{2}\|C^n\|_{\frac{r^-}{r^--1}}\|{\rm{div}}\,\UU^n\|_{r^-}\\
&\leq C\|\UU^n\|_{1,r^-}\|\nabla C^n\|_2+C\|\UU^n\|_{1,r^-}\|\nabla C^n\|_{\frac{2r^-}{3r^--2}}\\
&\leq C(\varepsilon)\|\UU^n\|_{1,r^-}^2+\varepsilon\|\nabla C^n\|_2^2.
\end{align*}
Therefore, by \eqref{q1} and \eqref{UE1}, we have
\begin{equation}\label{UE2}
\int_{\Omega}|\nabla C^n|^2+|\q_c(C^n,\nabla C^n,\DD\UU^n)|^2\dx \leq C_2,
\end{equation}
where $C_2$ is independent of $n$.

Now, by Sobolev embedding and the uniform estimates \eqref{UE1} and \eqref{UE2}, we have for sufficiently large $t>0$ and for $q>2$ sufficiently close to $2$,
\[\|C^n\UU^n\|^q_q\leq\|C^n\|^q_t\|\UU^n\|^q_{\frac{tq}{t-q}}\leq C\|C^n\|^q_{1,2}\|\UU^n\|^q_{1,r^-}\leq C.\]
Also if we set $s\defeq\frac{2q}{q+2}$, for $q>2$ sufficiently close to $2$,
\[\|\nabla C^n\cdot \UU^n\|^s_s\leq\|\nabla C^n\|^s_2\|\UU^n\|^s_{\frac{2s}{2-s}}\leq\|C^n\|^s_{1,2}\|\UU^n\|^s_q\leq C\|C^n\|^s_{1,2}\|\UU^n\|^s_{1,r^-}\leq C.\]
Then we can apply Corollary \ref{disDe} with $g=C^n\UU^n$ and $h=\nabla C^n\cdot \UU^n$. Hence for some $\alpha \in (0,1)$, we obtain the following uniform bound, independent of $n\in\mathbb{N}$:
\begin{equation}\label{UE3}
\|C^n\|_{C^{0,\alpha}(\overline{\Omega})}\leq C_3.
\end{equation}

Since $C^{0,\alpha}(\overline{\Omega})$ is compactly embedded in $C^{0,\tilde{\alpha}}(\overline{\Omega})$ for all $\tilde{\alpha} \in (0,\alpha)$, we have that
\[C^n\rightarrow c \qquad{\rm{strongly}}\,\,{\rm{in}}\,\,C^{0,\tilde{\alpha}}(\overline{\Omega}),\]
which implies that
\[r\circ C^n\rightarrow r\circ c \qquad{\rm{strongly}}\,\,{\rm{in}}\,\,C^{0,\beta}(\overline{\Omega})\]
for some $ \beta \in (0,1)$. We can therefore apply Proposition \ref{dis_infsup} with $r^n(x)\defeq r\circ C^n(x)$. By \eqref{Q_n1}, \eqref{B_uest} and H\"older's inequality,
\begin{align*}
\|P^n\|_{(r^n)'(\cdot)}&\leq C\sup_{0\neq \VV\in\V^n,\|\VV\|_{1,r^n(\cdot)}\leq1}\langle{\rm{div}}\,\VV,P^n\rangle\\
&\leq C\sup_{0\neq \VV\in\V^n,\|\VV\|_{1,r^n(\cdot)}\leq1}\bigg|\int_{\Omega}\SSS(C^n,\DD\UU^n)\cdot \DD \VV\dx+B_u[\UU^n,\UU^n,\VV]-\langle \boldsymbol{f},\VV\rangle\bigg|\\
&\leq C\sup_{0\neq \VV\in\V^n,\|\VV\|_{1,r^n(\cdot)}\leq1}\bigg(\|\SSS(C^n,\DD\UU^n)\|_{(r^n)'(\cdot)}\|\DD \VV\|_{r^n(\cdot)}+\|\UU^n\|^2_{1,r^n(\cdot)}\|\VV\|_{1,r^n(\cdot)}\\
&\qquad\qquad\qquad\qquad\qquad\,\,\,\,\,\,+\|\boldsymbol{f}\|_{(W^{1,r^-}_0(\Omega)^2)^*}\|\VV\|_{1,r^n(\cdot)}\bigg).
\end{align*}
Therefore, by \eqref{UE1}, we have
\begin{equation}\label{UE4}
\|P^n\|_{(r^n)'(\cdot)}\leq C_4,
\end{equation}
where $C_4$ is independent of $n\in\mathbb{N}$.

Using the bounds \eqref{UE1}--\eqref{UE4}, thanks to their independence of $n\in\mathbb{N}$, reflexivity of the relevant spaces and compact Sobolev embedding, we can extract (not relabelled) subsequences such that
\begin{align}
\UU^n & \rightharpoonup \uu && {\rm{weakly}}\,\,{\rm{in}}\,\,W^{1,r^-}_{0}(\Omega)^2, \label{conv1}\\
\UU^n & \rightarrow \uu && {\rm{strongly}}\,\,{\rm{in}}\,\,L^{2(1+\varepsilon)}(\Omega)^2,\,\,(\varepsilon>0),\label{conv2}\\
C^n & \rightharpoonup c && {\rm{weakly}}\,\,{\rm{in}}\,\,W^{1,2}(\Omega),\label{conv3}\\
C^n & \rightarrow c && {\rm{strongly}}\,\,{\rm{in}}\,\,C^{0,\tilde{\alpha}}(\overline{\Omega}),\label{conv4}\\
P^n & \rightharpoonup p && {\rm{weakly}}\,\,{\rm{in}}\,\,L^{(r^+)'}(\Omega),\label{conv5}\\
\SSS(C^n,\DD\UU^n) & \rightharpoonup\bar{\SSS} && {\rm{weakly}}\,\,{\rm{in}}\,\,L^{(r^+)'}(\Omega)^{2\times 2},\label{conv6}\\
\q_c(C^n,\nabla C^n,\DD\UU^n) & \rightharpoonup \bar{\q}_c && {\rm{weakly}}\,\,{\rm{in}}\,\,L^2(\Omega)^2.\label{conv7}
\end{align}

Before proceeding, we shall prove that the limit function $\uu$ is contained in the desired space $W^{1,r(c)}_{0}(\Omega)^d$. Since $C^n\rightarrow c$ in $C^{0,\tilde{\alpha}}(\overline{\Omega})$, and by the continuity of $r$,
\[\forall\,\varepsilon>0,\,\,\,\exists N\in\mathbb{N}\,\,\,{\rm{such}}\,\,\,{\rm{that}}\,\,\,n\geq N\,\,\,{\rm{implies}}\,\,\,|r(C^n)-r(c)|<\frac{\varepsilon}{\theta},\]
where $\theta>1$ is large enough to satisfy $r(c)-\frac{\theta+1}{\theta}\varepsilon>1.$ We can then deduce from the estimate above that
\begin{align*}
C & \geq\int_{\Omega}|\nabla \UU^n|^{r(C^n)}\dx\geq\int_{|\nabla \UU^n|\geq1}|\nabla \UU^n|^{r(C^n)}\dx\\
&\geq\int_{|\nabla \UU^n|\geq1}|\nabla \UU^n|^{r(C^n)-r(c)+r(c)-\varepsilon}\dx\geq
\int_{|\nabla \UU^n|\geq1}|\nabla \UU^n|^{r(c)-\frac{\theta+1}{\theta}\varepsilon}\dx.
\end{align*}
Then, after adding to the inequality the term $\int_{|\nabla \UU^n|<1}|\nabla \UU^n|^{r(c)-\frac{\theta+1}{\theta}\varepsilon}\dx$, which is bounded by some constant $\bar{C}\leq|\Omega|$, we obtain
\[C+\bar{C}\geq\int_{\Omega}|\nabla \UU^n|^{r(c)-\frac{\theta+1}{\theta}\varepsilon}\dx.\]
Again, we can extract a (not relabelled) subsequence such that
\[\UU^n\rightharpoonup \uu\,\,\,{\rm{weakly}}\,\,\,{\rm{in}}\,\,\,W_0^{1,r(c)-\frac{\theta+1}{\theta}\varepsilon}(\Omega)^2.\]
Thus by using the weak lower-semicontinuity of the norm function, we see that
\[\int_{\Omega}|\nabla \uu|^{r(c)-\frac{\theta+1}{\theta}\varepsilon}\dx\leq C,\]
and consequently, Fatou's Lemma with $\varepsilon\rightarrow0$ leads us to
\begin{equation}\label{maininclu}
\int_{\Omega}|\nabla \uu|^{r(c)} \dx\leq C,
\end{equation}
which implies that $\uu\in W^{1,r(c)}_{0}(\Omega)^2$ by Poincar\'e's inequality.
With the same argument as above we can also show that
\begin{equation}\label{maininclu2}
\int_{\Omega}|\bar{\SSS}|^{r'(c)}+|p|^{r'(c)}\dx\leq C.
\end{equation}

Next, we prove that the limit $\uu$ is also exactly divergence-free. Let us consider an arbitrary but fixed $q\in C^{\infty}_0(\Omega)$. Then, by \eqref{Q_n2},
\begin{align*}
0&=\int_{\Omega}(\Pi^n_{\Q}q)\,{\rm{div}}\,\UU^n\dx\\
&=\int_{\Omega}(\Pi^n_{\Q}q-q)\,{\rm{div}}\,\UU^n\dx+\int_{\Omega}q({\rm{div}}\,\UU^n-{\rm{div}}\,\uu)\dx+\int_{\Omega}q\,{\rm{div}}\,\uu\dx.
\end{align*}
The first term tends to zero by \eqref{q_conv}, \eqref{UE1} and the second term converges to zero by \eqref{conv1}. Therefore,
\[\int_{\Omega}q\,{\rm{div}}\,\uu\dx=0\qquad{\rm{for}}\,\,\,{\rm{any}}\,\,\,q\in C^{\infty}_0(\Omega),\]
which implies that ${\rm{div}}\, \uu = 0$ a.e. on $\Omega$. In this case, we can identify the limit of the convective term $B_u[\cdot,\cdot,\cdot]$ as follows. Let us choose an arbitrary function $\vv\in W^{1,\infty}_0(\Omega)^2$ for which we define $\VV^n\defeq\Pi^n_{\rm{div}}\vv\in\V^n$. Then, by \eqref{v_conv}, we have
\begin{equation}\label{VV_n}
\VV^n\rightarrow \vv\qquad{\rm{strongly}}\,\,\,{\rm{in}}\,\,\,W^{1,\sigma}_0(\Omega)^2\,\,\,{\rm{for}}\,\,\,\sigma\in[1,\infty).
\end{equation}
Also, by the restriction $r^->1$, we have the continuous embedding $W^{1,r^n(\cdot)}_0(\Omega)^2\hookrightarrow L^{2(1+\varepsilon)}(\Omega)^2$. Therefore, by \eqref{UE1} and \eqref{conv2},
\[\UU^n\otimes \UU^n\rightarrow \uu\otimes \uu\qquad{\rm{strongly}}\,\,\,{\rm{in}}\,\,\,L^{1+\varepsilon}(\Omega)^2.\]
This then enables us to identify the second part of the convective term
\[-\int_{\Omega}(\UU^n\otimes \UU^n)\cdot\nabla \VV^n\dx\rightarrow-\int_{\Omega}(\uu\otimes \uu)\cdot\nabla \vv\dx\qquad{\rm{as}}\,\,\,n\rightarrow\infty.\]

On the other hand, for $r^->\frac{4}{3}$, we have the continuous embedding $W^{1,r^n(\cdot)}_0(\Omega)^2\hookrightarrow L^{(r^-)'+\varepsilon}(\Omega)^2$; thus $\UU^n\cdot \VV^n\rightarrow \uu\cdot \vv$ strongly in $L^{(r^-)'}(\Omega)^2$. Indeed,
\begin{align*}
\|\UU^n\cdot \VV^n-\uu\cdot \vv\|_{(r^-)'}
&\leq\|(\VV^n-\vv)\UU^n+(\UU^n-\uu)\vv\|_{(r^-)'}\\
&\leq\|\VV^n-\vv\|_s\|\UU^n\|_{(r^-)'+\varepsilon}+\|\UU^n-\uu\|_{(r^-)'+\varepsilon}\|\vv\|_s\\
&\leq\|\VV^n-\vv\|_s\|\UU^n\|_{1,r^n(\cdot)}+\|\UU^n-\uu\|_{\frac{2r^-}{2-r^-}-\varepsilon}\|\vv\|_s
\end{align*}
for some $s\in(1,\infty)$. The first term tends to zero thanks to \eqref{v_conv}, \eqref{UE1} and the second term converges to zero by \eqref{conv1} in conjunction with a compact embedding theorem. Therefore, together with ${\rm{div}}\,\uu=0$, we have
\begin{align*}
\int_{\Omega}(\UU^n\otimes \VV^n)\cdot\nabla \UU^n\dx
&=-\int_{\Omega}(\UU^n\otimes \UU^n)\cdot\nabla \VV^n\dx+\int_{\Omega}({\rm{div}}\,\UU^n)\,\UU^n\cdot \VV^n\dx\\
&\rightarrow-\int_{\Omega}(\uu\otimes \uu)\cdot\nabla \vv\dx\qquad{\rm{as}}\,\,\,n\rightarrow\infty.
\end{align*}
Collecting these limits, we then deduce that
\begin{equation}\label{convectionconv}
\lim_{n\rightarrow\infty}B_u[\UU^n,\UU^n,\VV^n]=-\int_{\Omega}(\uu\otimes \uu)\cdot\nabla  \vv\dx.
\end{equation}

Now, we are ready to pass to the limit in the first equation. By linearity of the projection operator $\Pi^n_{\rm{div}}$ and by noting \eqref{Q_n1}, we obtain that
\begin{align*}
\langle{\rm{div}}\,\vv,P^n\rangle
&=\langle{\rm{div}}\,\VV^n,P^n\rangle+\langle{\rm{div}}\,(\vv-\VV^n),P^n\rangle\\
&=\int_{\Omega}\SSS(C^n,\DD\UU^n)\cdot \DD\VV^n \dx-\langle \boldsymbol{f},\VV^n\rangle+B_u[\UU^n,\UU^n,\VV^n]\\
&\,\,\,\,\,\,+\langle{\rm{div}}\,(\vv-\VV^n),P^n\rangle\\
&\rightarrow\int_{\Omega}\bar{\SSS}\cdot \DD\vv+{\rm{div}}(\uu\otimes \uu)\cdot \vv\dx-\langle \boldsymbol{f},\vv\rangle,
\end{align*}
where we have used \eqref{conv5}, \eqref{conv6}, \eqref{VV_n} and \eqref{convectionconv}. Also, by \eqref{conv5} again,
\[\langle{\rm{div}}\,\vv,P^n\rangle\rightarrow\langle{\rm{div}}\,\vv,p\rangle.\]
Altogether, we have
\begin{equation}\label{limiteq1}
\int_{\Omega}\bar{\SSS}\cdot \DD \vv+{\rm{div}}\,(\uu\otimes \uu)\cdot \vv\dx-\langle{\rm{div}}\,\vv,p\rangle=\langle \boldsymbol{f},\vv\rangle\qquad\forall\, \vv\in W^{1,\infty}_0(\Omega)^2.
\end{equation}
We note that by using the same argument as above we have that
\begin{equation}\label{limiteq1-2}
\int_{\Omega}\bar{\SSS}\cdot \DD \vv+{\rm{div}}\,(\uu\otimes \uu)\cdot \vv\dx=\langle \boldsymbol{f},\vv\rangle\qquad\forall\, \vv\in W^{1,\infty}_{0,{\rm{div}}}(\Omega)^2.
\end{equation}

Now, let us investigate the limit of the equation for the concentration, \eqref{Q_n3}. We fix an arbitrary $z\in W^{1,2}_0(\Omega)$ and define $Z^n\defeq\Pi^n_{\Z}z\in\Z^n$. Thanks to \eqref{conv2} and \eqref{conv4},
\begin{align*}
\|C^n\UU^n-c\uu\|_2
&\leq\|(C^n-c)\UU^n\|_2+\|c(\UU^n-\uu)\|_2\\
&\leq\|C^n-c\|_{\infty}\|\UU^n\|_{2(1+\varepsilon)}+\|c\|_{\infty}\|\UU^n-\uu\|_{2(1+\varepsilon)}\rightarrow0.
\end{align*}
Also, by \eqref{z_conv}, \eqref{conv2} and Sobolev embedding,
\begin{align*}
\|Z^n\UU^n-z\uu\|_2
&\leq\|(Z^n-z)\UU^n\|_2+\|z(\UU^n-\uu)\|_2\\
&\leq\|Z^n-z\|_{\frac{2(1+\varepsilon)}{\varepsilon}}\|\UU^n\|_{2(1+\varepsilon)}
+\|z\|_{\frac{2(1+\varepsilon)}{\varepsilon}}\|\UU^n-\uu\|_{2(1+\varepsilon)}\\
&\leq C\|Z^n-z\|_{1,2}\|\UU^n\|_{2(1+\varepsilon)}+C\|z\|_{1,2}\|\UU^n-\uu\|_{2(1+\varepsilon)}\rightarrow0.
\end{align*}
In other words,
\begin{align}
C^n\UU^n&\rightarrow c\uu\qquad{\rm{strongly}}\,\,\,{\rm{in}}\,\,\,L^2(\Omega)^2,\label{CU}\\
Z^n\UU^n&\rightarrow z\uu\qquad{\rm{strongly}}\,\,\,{\rm{in}}\,\,\,L^2(\Omega)^2.\label{ZU}
\end{align}
By \eqref{conv3} and \eqref{ZU},
\begin{align*}
&\bigg|\int_{\Omega}Z^n\UU^n\cdot\nabla C^n\dx-\int_{\Omega}z\uu\cdot\nabla c\dx\bigg|\\
&\qquad\leq\int_{\Omega}|Z^n\UU^n-z\uu||\nabla C^n|\dx+\bigg|\int_{\Omega}z\uu\cdot(\nabla C^n-\nabla c)\dx\bigg|\rightarrow0.
\end{align*}
Hence, because ${\rm{div}}\,\uu=0$ a.e. on $\Omega$, we have that
\[\int_{\Omega}Z^n\UU^n\cdot\nabla C^n\dx\rightarrow\int_{\Omega}z\uu\cdot\nabla c\dx=-\int_{\Omega}c\uu\cdot\nabla z\dx\qquad{\rm{as}}\,\,\,n\rightarrow\infty.\]
Additionally, by \eqref{z_conv} and \eqref{CU},
\begin{align*}
&\,\,\,\,\,\,\,\,\bigg|\int_{\Omega}C^n\UU^n\cdot\nabla Z^n\dx-\int_{\Omega}c\uu\cdot\nabla z\dx\bigg|\\
&\qquad \leq\|C^n\UU^n\|_2\|Z^n-z\|_{1,2}+\|C^n\UU^n-c\uu\|_2\|z\|_{1,2}\rightarrow0.
\end{align*}
Altogether, we have
\[\lim_{n\rightarrow\infty} B_c[C^n,\UU^n,Z^n]=-\int_{\Omega}c\uu\cdot\nabla z\dx.\]
Finally, from \eqref{z_conv} and \eqref{conv7}, we have
\[\int_{\Omega}\q_c(C^n,\nabla C^n,\DD\UU^n)\cdot\nabla Z^n\dx\rightarrow \int_{\Omega}\bar{\q}_c\cdot\nabla z\dx\qquad{\rm{as}}\,\,\,n\rightarrow\infty.\]
By collecting the limits of the two terms, we then have that
\begin{equation}\label{limiteq2}
\int_{\Omega}\bar{\q}_c\cdot\nabla z-c\uu\cdot\nabla z\dx=0\qquad\forall\, z\in W^{1,2}_0(\Omega).
\end{equation}
We see from \eqref{limiteq1} and \eqref{limiteq2} that all that remains to be shown is the identification of the limits:
\[\bar{\SSS}=\SSS(c,\DD\uu)\,\,\,\,\,{\rm{and}}\,\,\,\,\,\bar{\q}_c=\q_c(c,\nabla c, \DD\uu).\]
\end{subsection}

\begin{subsection}{Compactness of $\DD\UU^n$}
Our proof of the identification of the limits begins by showing the compactness of $\DD\UU^n$ in the sense that
\[\lim_{n\rightarrow\infty}\int_{\Omega}\left((\SSS(C^n,\DD\UU^n)-\SSS(C^n,\DD\uu))\cdot(\DD\UU^n-\DD\uu)\right)^{\frac{1}{4}}\dx=0.\]
By \eqref{S1}, \eqref{S2}, \eqref{UE1}, \eqref{maininclu} and H\"older's inequality, we see that
\begin{equation}\label{limsup}
0\leq\limsup_{n\rightarrow\infty}\int_{\Omega}\left((\SSS(C^n,\DD\UU^n)-\SSS(C^n,\DD\uu))\cdot(\DD\UU^n-\DD\uu)\right)^{\frac{1}{4}}\dx=L<\infty.
\end{equation}
Hence, it is enough to show that $L=0$. For arbitrary fixed $\chi>0$, define
\[\Omega_{\chi}\defeq\{x\in\Omega:|\DD\uu|>\chi\}.\]
Then by \eqref{maininclu}, we have
\[|\Omega_{\chi}|\leq\int_{\Omega}\frac{|\DD\uu|}{\chi}\dx\leq\frac{C}{\chi}.\]
Now we decompose the integral
\begin{equation}\label{decompose}
\int_{\Omega}\left((\SSS(C^n,\DD\UU^n)-\SSS(C^n,\DD\uu))\cdot(\DD\UU^n-\DD\uu)\right)^{\frac{1}{4}}\dx=A(n,\chi)+B(n,\chi),
\end{equation}
where
\begin{align*}
A(n,\chi)&\defeq\int_{\Omega_{\chi}}\left((\SSS(C^n,\DD\UU^n)-\SSS(C^n,\DD\uu))\cdot(\DD\UU^n-\DD\uu)\right)^{\frac{1}{4}}\dx,\\
B(n,\chi)&\defeq\int_{\Omega\setminus\Omega_{\chi}}\left((\SSS(C^n,\DD\UU^n)-\SSS(C^n,\DD\uu))\cdot(\DD\UU^n-\DD\uu)\right)^{\frac{1}{4}}\dx.
\end{align*}
First, by \eqref{S1}, \eqref{UE1}, \eqref{maininclu} and H\"older's inequality,
\[A(n,\chi)\leq C|\Omega_{\chi}|^{\frac{1}{2}}\leq\frac{C}{\sqrt{\chi}}.\]
Next, we introduce a matrix-truncation function $T_{\chi}:\R^{2\times 2}\rightarrow\R^{2\times 2}$ as

\begin{displaymath}
T_{\chi}(\boldsymbol{M})=\left\{ \begin{array}{ll}
\boldsymbol{M} &\textrm{for $|\boldsymbol{M}|\leq\chi$,}\\
\chi\frac{\boldsymbol{M}}{|\boldsymbol{M}|}& \textrm{for $|\boldsymbol{M}|>\chi$.}
\end{array}\right.
\end{displaymath}
Since $T_{\chi}(\DD\uu)=\DD\uu$ on $\Omega\setminus\Omega_{\chi}$ and the integrand is positive, we can rewrite $B(n,\chi)$ as
\begin{align*}
B(n,\chi)&=\int_{\Omega\setminus\Omega_{\chi}}\left((\SSS(C^n,\DD\UU^n)-\SSS(C^n,T_{\chi}(\DD\uu)))\cdot(\DD\UU^n-T_{\chi}(\DD\uu))\right)^{\frac{1}{4}}\dx\\
&\leq\int_{\Omega}\left((\SSS(C^n,\DD\UU^n)-\SSS(C^n,T_{\chi}(\DD\uu)))\cdot(\DD\UU^n-T_{\chi}(\DD\uu))\right)^{\frac{1}{4}}\dx.
\end{align*}

Since $r$ is a H\"older-continuous function and $C^n$ satisfies \eqref{conv4}, we can apply Theorem \ref{LTmain}. Therefore, for any $j\in\mathbb{N}$, we can find $\UU^n_j\in\V^n\subset W^{1,\infty}_0(\Omega)^2$. Then, by H\"older's inequality,
\begin{align*}
&B(n,\chi)
\leq \left(\int_{\{\UU^n_j=\UU^n\}}\left(\SSS(C^n,\DD\UU^n)-\SSS(C^n,T_{\chi}(\DD\uu)))\cdot(\DD\UU^n-T_{\chi}(\DD\uu)\right)\dx\right)^{\frac{1}{4}}|\Omega|^{\frac{3}{4}}\\
&\quad+\left(\int_{\{\UU^n_j\neq \UU^n\}}\left((\SSS(C^n,\DD\UU^n)-\SSS(C^n,T_{\chi}(\DD\uu)))\cdot(\DD\UU^n-T_{\chi}(\DD\uu))\right)^{\frac{1}{2}}\dx\right)^{\frac{1}{2}}|\{\UU^n_j\neq \UU^n\}|^{\frac{1}{2}}\\
&\qquad\quad\,\eqdef(B_j(n,\chi))^{\frac{1}{4}}|\Omega|^{\frac{3}{4}}+(\tilde{B}_j(n,\chi))^{\frac{1}{2}}|\{\UU^n_j\neq \UU^n\}|^{\frac{1}{2}}.
\end{align*}
First, by \eqref{main4.6}, \eqref{main4.12} and \eqref{main4.14}, we have
\[|\{\UU^n_j\neq \UU^n\}|=\|\chi_{\{\UU^n_j\neq \UU^n\}}\|_{L^1(\Omega)}\leq\int_{\R^2}\frac{M(\DD\UU^n)}{\kappa\lambda^n_j}\dx\leq\frac{C}{(2^j)^{2^j}},\]
and thus it follows from \eqref{UE1}, \eqref{maininclu} and H\"older's inequality that
\[(\tilde{B}_j(n,\chi))^{\frac{1}{2}}|\{\UU^n_j\neq \UU^n\}|^{\frac{1}{2}}\leq\frac{C}{2^j}.\]
Next, we can rewrite $B_j(n,\chi)$ as
\begin{align}
B_j(n,\chi)&=\int_{\Omega}(\SSS(C^n,\DD\UU^n)-\SSS(C^n,T_{\chi}(\DD\uu)))\cdot(\DD\UU^n_j-T_{\chi}(\DD\uu))\dx\label{Bint1}\\
&\,\,\,\,\,\,-\int_{\{\UU^n_j\neq \UU^n\}}(\SSS(C^n,\DD\UU^n)-\SSS(C^n,T_{\chi}(\DD\uu)))\cdot(\DD\UU^n_j-T_{\chi}(\DD\uu))\dx.\label{Bint2}
\end{align}
By \eqref{S1}, \eqref{main4.7}, \eqref{main4.13}, H\"older's inequality and Young's inequality, we can analyze the second term, appearing in \eqref{Bint2}:
\begin{align*}
&\bigg|\int_{\{\UU^n_j\neq \UU^n\}}(\SSS(C^n,\DD\UU^n)-\SSS(C^n,T_{\chi}(\DD\uu)))\cdot(\DD\UU^n_j-T_{\chi}(\DD\uu))\dx\bigg|\\
&\leq\int_{\{\UU^n_j\neq \UU^n\}}|\SSS(C^n,\DD\UU^n)\cdot \DD\UU^n_j|\dx+C(\chi)\int_{\{\UU^n_j\neq \UU^n\}}(|\SSS(C^n,\DD\UU^n)|+|\DD\UU^n_j|+1)\dx\\
&\leq C\int_{\{\UU^n_j\neq \UU^n\}}|\nabla \UU^n|^{r^n(x)-1}\lambda^n_j\dx+C(\chi)|\{\UU^n_j\neq \UU^n\}|^{\frac{1}{r^+}}+\frac{C(\chi)}{2^j}\\
&\leq\frac{C}{(r^+)'}\int_{\{\UU^n_j\neq \UU^n\}}|\nabla \UU^n|^{r^n(x)}\dx+\frac{C}{r^-}\int_{\{\UU^n_j\neq \UU^n\}}|\lambda^n_j|^{r^n(x)}\dx+\frac{C(\chi)}{2^j}\\
&\leq\frac{C(\chi)}{2^j}.
\end{align*}

Now, to analyze the first term \eqref{Bint1} above, we have to use the weak formulation. {\color{black}{Here, however, we cannot use the Lipschitz truncation $\UU^n_j$ as a test function, as it is not guaranteed to be discretely divergence-free. To overcome this difficulty, we shall define discretely divergence-free approximations with zero trace with the help of the discrete Bogovski\u{\i} operator; more precisely, let}}
\begin{align*}
\boldsymbol{\Psi}^n_j&\defeq\mathcal{B}^n({\rm{div}}\,\UU^n_j),\\
\boldsymbol{\Phi}^n_j&\defeq \UU^n_j-\boldsymbol{\Psi}^n_j.
\end{align*}
It is then clear that $\boldsymbol{\Phi}^n_j$ has a zero trace on $\partial\Omega$ and, by construction, $\boldsymbol{\Phi}^n_j\in\V^n_{\rm{div}}$. Moreover, from the compact embedding $W^{1,\sigma}_0(\Omega)\hookrightarrow\hookrightarrow L^{\sigma}(\Omega)$, \eqref{main4.9} and Lemma \ref{Bog_conv}, we have
\begin{align}
\boldsymbol{\Phi}^n_j&\rightharpoonup \UU_j-\mathcal{B}({\rm{div}}\,\UU_j)\eqdef\boldsymbol{\Phi}_j&&{\rm{weakly}}\,\,{\rm{in}}\,\, W^{1,\sigma}_0(\Omega)^2,\label{main5.17}\\
\boldsymbol{\Phi}^n_j&\rightarrow\boldsymbol{\Phi}_j&&{\rm{strongly}}\,\,{\rm{in}}\,\,L^{\sigma}(\Omega)^2,\label{main5.18}
\end{align}
as $n\rightarrow\infty$, where $\sigma\in(1,\infty)$ is arbitrary. We can then rewrite \eqref{Bint1} above in terms of this approximation to obtain
\begin{align*}
&\int_{\Omega}(\SSS(C^n,\DD\UU^n)-\SSS(C^n,T_{\chi}(\DD\uu)))\cdot(\DD\UU^n_j-T_{\chi}(\DD\uu))\dx\\
&=\int_{\Omega}\SSS(C^n,\DD\UU^n)\cdot(\DD\boldsymbol{\Phi}^n_j+\DD\boldsymbol{\Psi}^n_j)\dx\\
&\,\,\,\,\,\,-\int_{\Omega}\SSS(C^n,\DD\UU^n)\cdot T_{\chi}(\DD\uu)\dx-\int_{\Omega}\SSS(C^n,T_{\chi}(\DD\uu))\cdot(\DD\UU^n_j-T_{\chi}(\DD\uu))\dx\\
&\eqdef B^{n,1}_{\chi,j}-B^{n,2}_{\chi,j}-B^{n,3}_{\chi,j}.
\end{align*}
Now we use \eqref{P_n1} with $\VV=\boldsymbol{\Phi}^n_j\in\V^n_{\rm{div}}$ and pass to the limit with \eqref{conv2}, \eqref{conv6}, and \eqref{main5.17}; thus we have, by \eqref{limiteq1-2}, that
\begin{align*}
\lim_{n\rightarrow\infty}\int_{\Omega}\SSS(C^n,\DD\UU^n)\cdot \DD\boldsymbol{\Phi}^n_j\dx
&=-\lim_{n\rightarrow\infty}B_u[\UU^n,\UU^n,\boldsymbol{\Phi}^n_j]+\lim_{n\rightarrow\infty}\langle \boldsymbol{f},\boldsymbol{\Phi}^n_j\rangle\\
&=\int_{\Omega}(\uu\otimes \uu)\cdot\nabla\boldsymbol{\Phi_j}\dx+\langle \boldsymbol{f},\boldsymbol{\Phi_j}\rangle\\
&=\int_{\Omega}\bar{\SSS}\cdot \DD\boldsymbol{\Phi_j}\dx.
\end{align*}

Let us now consider the second integral in $B^{n,1}_{\chi,j}$. Using the boundedness of $\SSS(C^n,\DD\UU^n)$ in $L^{r'(C^n)}(\Omega)^{2\times 2}$, we can estimate it by H\"older's inequality as follows:
\[\int_{\Omega}\SSS(C^n,\DD\UU^n)\cdot \DD\boldsymbol{\Psi}^n_j\dx
\leq C\|\DD\boldsymbol{\Psi}^n_j\|_{r^n(\cdot)}\leq C\|\Pi^n_{\rm{div}}\mathcal{B}\mathcal{K}({\rm{div}}\,\UU^n_j)\|_{1,r^n(\cdot)}.\]

By \eqref{Bog_est}, and Theorem \ref{contibog},
\[\|\mathcal{B}\mathcal{K}({\rm{div}}\,\UU^n_j)\|_{1,r^n(\cdot)}\leq C\|\mathcal{K}({\rm{div}}\,\UU^n_j)\|_{r^n(\cdot)}\leq C\sup_{Q\in\Q^n,\,\,\|Q\|_{(r^n)'(\cdot)}\leq1}\langle{\rm{div}}\,\UU^n_j\,Q\rangle.\]
We deduce, by H\"older's inequality, that
\begin{align*}
\langle{\rm{div}}\,\UU^n_j,Q\rangle
&=\sum_{E\subset\{\UU^n_j=\UU^n\}}\langle{\rm{div}}\,\UU^n,\chi_EQ\rangle
+\sum_{E\cap\{\UU^n_j\neq \UU^n\}\neq\emptyset}\langle{\rm{div}}\,\UU^n_j,\chi_EQ\rangle\\
&\leq\norm{{\rm{div}}\,\UU^n_j\chi_{S_{\{\UU^n_j\neq \UU^n\}}}}_{r^n(\cdot)}\norm{\sum_{E\cap\{\UU^n_j\neq \UU^n\}\neq\emptyset}\chi_EQ}_{(r^n)'(\cdot)}\\
&\leq\norm{\nabla \UU^n_j\chi_{S_{\{\UU^n_j\neq \UU^n\}}}}_{r^n(\cdot)}\norm{\sum_{E\cap\{\UU^n_j\neq \UU^n\}\neq\emptyset}\chi_EQ}_{(r^n)'(\cdot)},
\end{align*}
where $\chi_{S_{\{\UU^n_j\neq \UU^n\}}}$ is the characteristic function of the set
\[{S_{\{\UU^n_j\neq \UU^n\}}}\defeq\bigcup\left\{S_E:E\in\G_n\,\,\,{\rm{such}}\,\,\,{\rm{that}}\,\,\,E\cap\overline{\{\UU^n_j\neq \UU^n\}}\neq\emptyset\right\}.\]
Then, by Lemma \ref{DLTL2} and \eqref{main4.13},
\[\norm{\nabla \UU^n_j\chi_{S_{\{\UU^n_j\neq \UU^n\}}}}_{r^n(\cdot)}\leq\frac{C}{2^{j/r^+}}.\]
Also, by Theorem \ref{local-global},
\begin{align*}
\norm{\sum_{E\cap\{\UU^n_j\neq \UU^n\}\neq\emptyset}\chi_EQ}_{(r^n)'(\cdot)}
&\leq C\norm{\sum_{E\cap\{\UU^n_j\neq \UU^n\}\neq\emptyset}\chi_E\frac{\|\chi_EQ\|_{(r^n)'(\cdot)}}{\|\chi_E\|_{(r^n)'(\cdot)}}}_{(r^n)'(\cdot)}\\
&\leq C\norm{\sum_{E\in\G_n}\chi_E\frac{\|\chi_EQ\|_{(r^n)'(\cdot)}}{\|\chi_E\|_{(r^n)'(\cdot)}}}_{(r^n)'(\cdot)}\\
&\leq C\norm{\sum_{E\in\G_n}\chi_EQ}_{(r^n)'(\cdot)}\\
&\leq C \|Q\|_{(r^n)'(\cdot)}.
\end{align*}
Therefore, we have
\[\|\mathcal{B}\mathcal{K}({\rm{div}}\,\UU^n_j)\|_{1,r^n(\cdot)}\leq \frac{C}{2^{j/r^+}},\]
which implies, together with Proposition \ref{sta_variable}, that
\[\|\Pi^n_{\rm{div}}\mathcal{B}\mathcal{K}({\rm{div}}\,\UU^n_j)\|_{1,r^n(\cdot)}\leq
\left(\frac{C}{2^{j/r^+}}+ C\max_{E\in\G_n}h_E^{d+1} \right)^{\gamma}\]
for some $\gamma=\gamma(r^-,r^+)>0$.

Now, note further that by weak lower-semicontinuity and boundedness of $\bar{\SSS}$ in $L^{r'(c)}$,
\begin{align}
\int_{\Omega}\bar{\SSS}\cdot \DD\mathcal{B}({\rm{div}}\,\UU_j)\dx
&\leq C\|\mathcal{B}({\rm{div}}\,\UU_j)\|_{1,r(c)}
\leq C\limsup_{n\rightarrow\infty}\|\mathcal{B}^n({\rm{div}}\,\UU^n_j)\|_{1,r^n(\cdot)}\leq \left(\frac{C}{2^{j/r^+}}\right)^{\gamma}.\label{second_conv}
\end{align}
For the last two integrals $B^{n,2}_{\chi,j}$ and $B^{n,3}_{\chi,j}$, we use \eqref{main4.9}, \eqref{conv4}, \eqref{conv6} and the boundedness of the truncation $T_{\chi}$ to get
\[\lim_{n\rightarrow\infty}\left(B^{n,2}_{\chi,j}+B^{n,3}_{\chi,j}\right)=\int_{\Omega}
\bar{\SSS}\cdot T_{\chi}(\DD\uu)\dx+\int_{\Omega}\SSS(c,T_{\chi}(\DD\uu))\cdot(\DD\UU_j-T_{\chi}(\DD\uu))\dx.\]
Altogether, we have
\begin{align*}
&\lim_{n\rightarrow\infty}\left(B^{n,1}_{\chi,j}-B^{n,2}_{\chi,j}-B^{n,3}_{\chi,j}\right)
\leq\int_{\Omega}\bar{\SSS}\cdot \DD\boldsymbol{\Phi}_j\dx+\left(\frac{C}{2^{j/r^+}}\right)^{\gamma}-\lim_{n\rightarrow\infty}\left(B^{n,2}_{\chi,j}+B^{n,3}_{\chi,j}\right)\\
&\quad=\int_{\Omega}\bar{\SSS}\cdot \DD\UU_j\dx-\int_{\Omega}\bar{\SSS}\cdot \DD\mathcal{B}({\rm{div}}\,\UU_j)\dx+\left(\frac{C}{2^{j/r^+}}\right)^{\gamma}-\lim_{n\rightarrow\infty}\left(B^{n,2}_{\chi,j}+B^{n,3}_{\chi,j}\right)\\
&\quad\leq\int_{\Omega}(\bar{\SSS}-\SSS(c,T_{\chi}(\DD\uu)))\cdot(\DD\UU_j-T_{\chi}(\DD\uu))\dx+\left(\frac{C}{2^{j/r^+}}\right)^{\gamma}.
\end{align*}

Going back to \eqref{decompose}, we finally let $\chi$, $j\rightarrow\infty$ and $n\rightarrow\infty$, and estimate
\begin{align*}
&\lim_{\chi\rightarrow\infty}\lim_{j\rightarrow\infty}\lim_{n\rightarrow\infty}(A(n,\chi)+B(n,\chi))\\
&\quad\leq\lim_{\chi\rightarrow\infty}\lim_{j\rightarrow\infty}\lim_{n\rightarrow\infty}
\left(C\left(B^{n,1}_{\chi,j}-B^{n,2}_{\chi,j}-B^{n,3}_{\chi,j}+\frac{C(\chi)}{2^j}\right)^{\frac{1}{4}}|\Omega|^{\frac{3}{4}}+\frac{C}{\sqrt{\chi}}+\frac{C}{2^j}\right)\\
&\quad\leq\lim_{\chi\rightarrow\infty}C\left(\left(\int_{\Omega}(\bar{\SSS}-\SSS(c,T_{\chi}(\DD\uu)))\cdot(\DD\uu-T_{\chi}(\DD\uu))\dx\right)^{\frac{1}{4}}+\frac{C}{\sqrt{\chi}}\right)=0,
\end{align*}
where we have used \eqref{vjconv} for $j\rightarrow\infty$ and the pointwise convergence of $T_{\chi}(\DD\uu)\rightarrow \DD\uu$ on $\Omega$ with the dominated convergence theorem for $\chi\rightarrow\infty$. We have thereby completed the proof of the desired compactness of $\DD\UU^n$.
\end{subsection}

\begin{subsection}{Identification of $\bar{\SSS}=\SSS(c,\DD\uu)$ and $\bar{\q}_c=\q_c(c,\nabla c,\DD\uu)$}
In the previous section we showed that
\begin{equation}\label{cptness}
\lim_{n\rightarrow\infty}\int_{\Omega}\left((\SSS(C^n,\DD\UU^n)-\SSS(C^n,\DD\uu))\cdot(\DD\UU^n-\DD\uu)\right)^{\frac{1}{4}}\dx=0.
\end{equation}
Since the integrand is nonnegative, \eqref{cptness} also holds for a set $Q_{\gamma}\subset\Omega$ where
\[Q_{\gamma}\defeq\{x\in\Omega:|\DD\uu|\leq\gamma\},\]
with an arbitrarily fixed constant $\gamma>0$. From the sequence of integrands of \eqref{cptness}, we can extract a subsequence (again not relabelled), which converges to zero almost everywhere in $Q_{\gamma}$. Then, by Egoroff's theorem, for arbitrary $\varepsilon>0$, we can find a set $Q^{\varepsilon}_{\gamma}\subset\Omega$ such that $|Q_{\gamma}\setminus Q^{\varepsilon}_{\gamma}|<\varepsilon$, where the sequence of integrands converges uniformly. It is obvious that, thanks to the choice of $Q^{\varepsilon}_{\gamma}$, we have
\[\lim_{\gamma\rightarrow\infty}\lim_{\varepsilon\rightarrow0}|\Omega\setminus Q^{\varepsilon}_{\gamma}|=0,\]
and furthermore, from the uniform convergence, we have
\begin{equation}\label{main5.21}
\lim_{n\rightarrow\infty}\int_{Q^{\varepsilon}_{\gamma}}(\SSS(C^n,\DD\UU^n)-\SSS(C^n,\DD\uu))\cdot (\DD\UU^n-\DD\uu)\dx=0.
\end{equation}
Thanks to the boundedness of $\DD\uu$ on $Q^{\varepsilon}_{\gamma}$, by the dominated convergence theorem, we have $\SSS(C^n,\DD\uu)\rightarrow \SSS(c,\DD\uu)$ strongly in $L^q(\Omega)^{2\times 2}$ for any $q\in[1,\infty)$.
Thus, together with the above $L^q$-convergence and weak convergence \eqref{conv1}, from \eqref{main5.21}, we have
\[\lim_{n\rightarrow\infty}\int_{Q^{\varepsilon}_{\gamma}}\SSS(C^n,\DD\UU^n)\cdot(\DD\UU^n-\DD\uu)\dx=0.\]
Hence, by the boundedness of $\DD\uu$ on $Q^{\varepsilon}_{\gamma}$ and the convergence result \eqref{conv5}, we have
\begin{equation}\label{main5.22}
\lim_{n\rightarrow\infty}\int_{Q^{\varepsilon}_{\gamma}}\SSS(C^n,\DD\UU^n)\cdot \DD\UU^n\dx=\int_{Q^{\varepsilon}_{\gamma}}\bar{\SSS}\cdot \DD\uu\dx.
\end{equation}
Now, let $\boldsymbol{B}\in L^{\infty}(Q^{\varepsilon}_{\gamma})^{2\times 2}$ be arbitrarily fixed. By the monotonicity assumption \eqref{S2},
\begin{equation}\label{main5.23}
0\leq\int_{Q^{\varepsilon}_{\gamma}}(\SSS(C^n,\DD\UU^n)-\SSS(C^n,\boldsymbol{B}))\cdot(\DD\UU^n-\boldsymbol{B})\dx.
\end{equation}
Thus, from \eqref{main5.22}, the $L^q$-convergence of $\SSS(C^n,\boldsymbol{B})\rightarrow \SSS(c,\boldsymbol{B})$ and the weak convergence \eqref{conv1}, we have
\begin{align*}
0&\leq\lim_{n\rightarrow\infty}\int_{Q^{\varepsilon}_{\gamma}}(\SSS(C^n,\DD\UU^n)-\SSS(C^n,\boldsymbol{B}))\cdot(\DD\UU^n-\boldsymbol{B})\dx\\
&=\int_{Q^{\varepsilon}_{\gamma}}\bar{\SSS}\cdot(\DD\uu-\boldsymbol{B})\dx-\int_{Q^{\varepsilon}_{\gamma}}\SSS(c,\boldsymbol{B})\cdot(\DD\uu-\boldsymbol{B})\dx\\
&=\int_{Q^{\varepsilon}_{\gamma}}(\bar{\SSS}-\SSS(c,\boldsymbol{B}))\cdot(\DD\uu-\boldsymbol{B})\dx.
\end{align*}
Now we use Minty's trick. First, choose $\boldsymbol{B}=\DD\uu\pm\lambda \boldsymbol{A}(x)$ with $\lambda>0$ and $\boldsymbol{A}\in L^{\infty}(Q^{\varepsilon}_{\gamma})^{2\times 2}$. Then, passing to the limit $\lambda\rightarrow0$, the continuity of $\SSS$ gives us
\[\int_{Q^{\varepsilon}_{\gamma}}(\bar{\SSS}-\SSS(c,\DD\uu))\cdot \boldsymbol{A}(x)\dx=0.\]
Therefore, we have
\[\bar{\SSS}=\SSS(c,\DD\uu)\,\,\,{\rm{a.e.}}\,\,\,{\rm{on}}\,\,\,Q^{\varepsilon}_{\gamma}.\]
So now we let $\varepsilon\rightarrow0$ and then $\gamma\rightarrow\infty$ to conclude that
\[\bar{\SSS}=\SSS(c,\DD\uu)\,\,\,{\rm{a.e.}}\,\,\,{\rm{on}}\,\,\,\Omega.\]

Finally, since $\SSS$ is strictly monotonic and $C^n\rightarrow c$ in $C^{0,\tilde{\alpha}}(\overline{\Omega})$, from \eqref{cptness} we have
\begin{equation}\label{a.e.conv}
\DD\UU^n\rightarrow \DD\uu\,\,\,{\rm{a.e.}}\,\,\,{\rm{on}}\,\,\,\Omega.
\end{equation}
As a continuous linear operator preserves weak convergence, by the dominated convergence theorem with \eqref{conv3}, \eqref{conv4} and \eqref{a.e.conv}, we can deduce that
\[\q_c(C^n,\nabla C^n,\DD\UU^n)\rightharpoonup \q_c(c,\nabla c, \DD\uu)\qquad{\rm{weakly}}\,\,{\rm{in}}\,\,L^2(\Omega)^2.\]
Therefore, by the uniqueness of the weak limit, we can identify
\[\bar{\q}_c=\q_c(c,\nabla c,\DD\uu),\]
thus completing the proof of the convergence of the finite element method under consideration to a weak solution of the problem.
\end{subsection}
\end{section}

\begin{section}{Conclusion}
We have established the convergence of finite element approximations to a chemically reacting incompressible non-Newtonian fluid flow model in a two-dimensional convex polygonal domain. The model consists of a convection-diffusion equation for the concentration and a generalized Navier--Stokes equation, where the viscosity depends on the shear-rate and the concentration. Our key technical tools included discrete counterparts of the Bogovski\u{\i} operator, De Giorgi's regularity theorem and the Acerbi--Fusco Lipschitz truncation of Sobolev functions, which were used in combination with a variety of results in variable-exponent Lebesgue and Sobolev spaces.

An interesting direction for future research is the extension of the results obtained herein to unsteady models, including both the proof of the existence of a weak solution to the unsteady model, and the convergence of a fully discrete approximation to the model based on a mixed finite element discretization with respect to the spatial variables. A nontrivial open problem is the extension of the two-dimensional discrete De Giorgi estimate to the case of three space dimensions. The argument used here in the case of two space dimensions relied on a discrete counterpart of Meyers' regularity theorem in conjunction with Morrey's embedding theorem.
This kind of argument for deriving a uniform H\"older norm bound on the sequence of approximate
solutions to the concentration equation is specific to the case of $d=2$.
The extension of the analysis developed here to the case of $d=3$ space dimensions is the subject of ongoing research
and will be discussed in a forthcoming paper, \cite{KS17}.
\end{section}

\section*{Acknowledgements}
Seungchan Ko's work was supported by the UK Engineering and Physical Sciences Research Council [EP/L015811/1]. Endre S\"uli is grateful to the Ne\v{c}as Center for Mathematical Modeling at the Faculty of Mathematics and Physics of the Charles University in Prague for the stimulating research environment during his sabbatical leave.

\bibliography{references}
\bibliographystyle{abbrv}


\end{document}